\documentclass[11pt,a4paper]{amsart}
\usepackage[margin=1in]{geometry}
\usepackage{amsmath,amssymb,amsthm}
\usepackage{booktabs}
\usepackage[colorlinks,linkcolor=blue,citecolor=blue]{hyperref}
\usepackage{graphicx}
\usepackage{float}

\newtheorem{thm}{Theorem}[section]
\numberwithin{thm}{section}
\newtheorem{lemma}[thm]{Lemma}
\newtheorem{remark}[thm]{Remark}
\newtheorem{proposition}[thm]{Proposition}
\newtheorem{corollary}[thm]{Corollary}
\numberwithin{equation}{section}
\numberwithin{table}{section}

\DeclareMathOperator{\Area}{Area}
\DeclareMathOperator{\Vol}{Vol}
\newcommand{\RR}{\mathbb{R}}
\newcommand{\PP}{\ensuremath{\mathrm{P}}}
\newcommand{\DD}{\ensuremath{\mathrm{D}}}
\newcommand{\Tet}{\mathcal{Q}_{\mathrm{tet}}}
\newcommand{\Oct}{\mathcal{Q}_{\mathrm{oct}}}
\newcommand{\KK}{K\!\left(\tfrac14\right)}
\newcommand{\KKp}{K\!\left(\tfrac34\right)}
\newcommand{\At}{\widetilde{A}_{3}}
\newcommand{\Bt}{\widetilde{B}_{3}}
\newcommand{\Ct}{\widetilde{C}_{3}}
\newcommand{\Qu}{\mathcal{T}}        

\title[Two Smyth disks in a tetragonal disphenoid]
{Schwarz \PP{} and Schwarz \DD{} as free--boundary minimal disks\\
in a tetragonal disphenoid: exact areas}
\author{Steven Finch}
\address{MIT Sloan School of Management, Cambridge, MA 02139, USA}
\email{steven\_finch\_math@outlook.com}
\date{\today}
\subjclass[2020]{Primary 53A10; Secondary 49Q05, 33E05, 20F55}
\keywords{Schwarz D surface, Schwarz P surface, conjugate minimal surface,
tetragonal disphenoid, Smyth disk, free boundary, affine Weyl group,
elliptic integral}

\begin{document}
\maketitle

\begin{abstract}
Let $T$ be the tetragonal disphenoid whose four congruent faces are
isosceles triangles with sides $2,\sqrt3,\sqrt3$, so that the dihedral
angles are $\pi/2$ along one opposite pair of edges and $\pi/3$ along the
remaining four. By Smyth's theorem $T$ contains three embedded minimal
disks meeting $\partial T$ orthogonally; the symmetry group $D_{2d}$ of
$T$ fixes one and interchanges the other two. We identify all three
completely. The distinguished disk is a patch of \textbf{Schwarz's
\PP{} surface} and the congruent pair are patches of \textbf{Schwarz's
\DD{} surface}, with exact areas
\[
\Area(\Sigma^{*}_{A})=\frac{\KK}{\KKp}=0.781700961348056\ldots,
\
\Area(\Sigma^{*}_{B})=\frac{\KKp}{2\,\KK}=0.639630785585503\ldots
\]
in the normalization $\Vol(T)=2/3$, inradius $r_T=\sqrt2/4$. In
particular $2\,\Area(\Sigma^{*}_{A})\Area(\Sigma^{*}_{B})=1$, which we
show is the Bonnet relation for the adjoint pair $(\DD,\PP)$. The proof
is a reversal of the usual strategy: rather than uniformizing the Gauss
map, we identify Smyth's three Plateau contours as the \emph{tetrahedral
four edges} and the \emph{octahedral four edges} --- classically solved
by \DD{} and \PP{} respectively --- and then conjugate. Both areas agree
with Surface Evolver computations of Ken~Brakke to all $15$ reported
digits. We also show that all twelve corners of the three disks are
midpoints of edges of $T$, each of the six midpoints being used twice,
and that $T$ is precisely the alcove of the affine Weyl group $\widetilde{A}_{3}$, 
confirming a remark of Gandy, Cvijovi\'c, et al. (1999).
\end{abstract}

\section{Introduction and statement of results}

\subsection{The disphenoid}
Throughout, $T\subset\RR^{3}$ is the tetrahedron with vertices
\begin{equation}\label{eq:verts}
v_{1}=\left(1,0,\tfrac12\right),\quad
v_{2}=\left(-1,0,\tfrac12\right),\quad
v_{3}=\left(0,1,-\tfrac12\right),\quad
v_{4}=\left(0,-1,-\tfrac12\right).
\end{equation}
Then $|v_{1}v_{2}|=|v_{3}v_{4}|=2$ and the remaining four edges have
length $\sqrt3$, so each face is an isosceles triangle with sides
$2,\sqrt3,\sqrt3$ and area $\sqrt2$; $T$ is a tetragonal disphenoid
(isohedral, four congruent faces). Its centroid and incenter are both
the origin, and
\begin{equation}\label{eq:Tdata}
\Vol(T)=\frac23,\qquad
S(T)=4\sqrt2,\qquad
r_{T}=\frac{3\Vol(T)}{S(T)}=\frac{1}{2\sqrt2}=\frac{\sqrt2}{4}.
\end{equation}
The outward face normals are
\begin{equation}\label{eq:nrms}
\nu_{1}=\tfrac{1}{\sqrt2}(-1,0,-1),\quad 
\nu_{2}=\tfrac{1}{\sqrt2}(1,0,-1), \quad 
\nu_{3}=\tfrac{1}{\sqrt2}(0,-1,1),\quad 
\nu_{4}=\tfrac{1}{\sqrt2}(0,1,1),
\end{equation}
labelled so that $\nu_{i}$ is normal to the face $\varphi_{i}$ opposite vertex $v_{i}$; explicitly
$\nu_{1}\leftrightarrow v_{2}v_{3}v_{4}=\varphi_{1}$,
$\nu_{2}\leftrightarrow v_{1}v_{3}v_{4}=\varphi_{2}$,
$\nu_{3}\leftrightarrow v_{1}v_{2}v_{4}=\varphi_{3}$,
$\nu_{4}\leftrightarrow v_{1}v_{2}v_{3}=\varphi_{4}$.
One has $\sum_{i}\nu_{i}=0$ and
\begin{equation}\label{eq:gram}
\nu_{1}\!\cdot\!\nu_{2}=\nu_{3}\!\cdot\!\nu_{4}=0,
\qquad
\nu_{i}\!\cdot\!\nu_{j}=-\tfrac12\ \ \text{for all four cross pairs.}
\end{equation}
Since the dihedral angle $\delta_{ij}$ along the edge shared by faces
$i,j$ satisfies $\nu_{i}\!\cdot\!\nu_{j}=-\cos\delta_{ij}$, we read off
$\delta_{12}=\delta_{34}=\pi/2$ (the two edges of length $2$) and
$\delta_{ij}=\pi/3$ (the four edges of length $\sqrt3$). In the notation
of \cite{FinchReg} this is the member
\begin{equation}
\beta=1-\frac{\delta_{1}}{\pi}=\frac23,
\qquad \mu=1,
\qquad \cos\pi\beta=-\frac12,
\end{equation}
of the one-parameter family of disphenoids, i.e.,\ the endpoint
$\beta=2/3$ left unresolved there.

\subsection{Smyth's construction} \label{subsec:smyth}
Smyth \cite{Smyth} proved that every tetrahedron contains three embedded
minimal disks meeting the boundary orthogonally, one for each of the three
ways of arranging the four face-normals into a closed quadrilateral. They 
are obtained as follows. Put $c_{i}=\Area(\varphi_{i})$ for the four faces;
then $\sum_{i}c_{i}\nu_{i}=0$, so the four vectors $c_{i}\nu_{i}$ close
up into a skew quadrilateral $Q$ once an ordering is chosen. There are
three orderings up to rotation and reflection of the cycle. For each,
one solves the Plateau problem in $Q$, obtaining a disk $\Sigma$ bounded
by four straight segments, and takes the conjugate (adjoint) surface
$\Sigma^{*}$. By the boundary dictionary --- conjugation sends a
straight boundary line of direction $v$ to a planar geodesic in a plane
of normal $v$, met orthogonally --- $\Sigma^{*}$ has four planar free
boundary arcs lying in four planes with normals $\nu_{i}$, hence in the
faces of a tetrahedron homothetic to $T$. Smyth also proved the area
formula used repeatedly below:
\begin{equation}\label{eq:smytharea}
\Area=\tfrac12\,\mathcal{P}\,\rho,
\end{equation}
$\mathcal{P}$ the perimeter of the free boundary and $\rho$ the inradius
of the ambient tetrahedron.

For $T$ all four faces have equal area $c=\sqrt2$, so all three
quadrilaterals are \emph{equilateral with side $\sqrt2$}. Since the
interior angle of $Q$ at the vertex between consecutive edge vectors
$c\nu_{i},c\nu_{j}$ equals $\pi-\angle(\nu_{i},\nu_{j})$, equation
\eqref{eq:gram} gives interior angles $\pi/3$ when
$\nu_{i}\!\cdot\!\nu_{j}=-\tfrac12$ and $\pi/2$ when
$\nu_{i}\!\cdot\!\nu_{j}=0$.

\subsection{Terminology} \label{subsec:terminology}
Following \cite{FinchDP} we use two names, and we deliberately avoid the
phrase ``regular skew quadrilateral,'' which is ambiguous precisely
because \emph{both} of the following contours are equilateral.

\medskip\noindent
\textbf{The tetrahedral four edges} $\Tet$ is the closed polygon
\begin{equation}\label{eq:tet}
(0,0,0)\to(1,0,1)\to(1,1,0)\to(0,1,1)\to(0,0,0),
\end{equation}
four of the six edges of a regular tetrahedron inscribed in the unit
cube. All sides are $\sqrt2$, all interior angles are $\pi/3$, and both
diagonals are $\sqrt2$; equivalently all six mutual distances of the four
vertices are equal.

\medskip\noindent
\textbf{The octahedral four edges} $\Oct$ is the closed polygon
\begin{equation}\label{eq:oct}
\left(0,0,\tfrac12\right)\to\left(\tfrac12,-\tfrac12,\tfrac12\right)
\to\left(\tfrac12,0,1\right)\to\left(\tfrac12,\tfrac12,\tfrac12\right)
\to\left(0,0,\tfrac12\right),
\end{equation}
four edges of a regular octahedron. All sides are $1/\sqrt2$; the
interior angles alternate $\pi/2,\pi/3,\pi/2,\pi/3$; the diagonals are
$1/\sqrt2$ and $1$.

\medskip\noindent
The classical solutions, with the areas as computed in \cite{GKMK,GK}, are
\begin{align}
\Area(\text{Plateau disk in }\Tet)&=\frac{\KKp}{\KK}=1.2792615711\ldots
&&\text{a patch of }\DD,\label{eq:areatet}\\
\Area(\text{Plateau disk in }\Oct)&=\frac12\,\frac{\KK}{\KKp}
=0.3908504806\ldots &&\text{a patch of }\PP.\label{eq:areaoct}
\end{align}
Here $K(m)$ is the complete elliptic integral of the first kind with
parameter $m$ (Mathematica's \texttt{EllipticK}); we abbreviate
$K=\KK$, $K'=\KKp$. It is convenient to record the scale-invariant
shape constants
\begin{equation}\label{eq:kappas}
\kappa_{Q}:=\frac{\Area(\text{Plateau disk in }Q)}{(\text{side of }Q)^{2}},
\
\kappa_{\rm tet}=\frac{K'}{2K}=0.6396307856\ldots,
\
\kappa_{\rm oct}=\frac{K}{K'}=0.7817009613\ldots
\end{equation}
and to note the relation, used in \S\ref{sec:product},
\begin{equation}\label{eq:kprod}
\kappa_{\rm tet}\,\kappa_{\rm oct}=\frac12 .
\end{equation}

\subsection{Main theorem}

\begin{thm}\label{thm:main}
Let $T$ be as in \eqref{eq:verts} and let $\Sigma^{*}_{A}$,
$\Sigma^{*}_{B}$, $\Sigma^{*}_{C}$ be Smyth's three free-boundary
minimal disks, indexed so that $\Sigma^{*}_{A}$ is the one invariant
under the full symmetry group $D_{2d}$ of $T$. Then{\rm :}
\begin{enumerate}
\item[\textup{(i)}] The Plateau contour of $\Sigma^{*}_{A}$ is congruent
to the tetrahedral four edges $\Tet$ \textup{(}same scale\textup{)}, and
those of $\Sigma^{*}_{B},\Sigma^{*}_{C}$ are congruent to the octahedral
four edges $\Oct$ scaled by $2$.
\item[\textup{(ii)}] $\Sigma^{*}_{A}$ is a patch of the Schwarz
\PP{} surface, consisting of four of Gandy--Klinowski's fundamental
patches, in a simple cubic lattice of edge $2$.
\item[\textup{(iii)}] $\Sigma^{*}_{B}$ and $\Sigma^{*}_{C}$ are
congruent patches of the Schwarz \DD{} surface, each consisting of two
fundamental patches, in a face-centered cubic (fcc) lattice of conventional cube edge $4$.
\item[\textup{(iv)}] The areas are
\begin{equation}\label{eq:areas}
\Area(\Sigma^{*}_{A})=\frac{K}{K'}
=0.781700961348056\ldots,
\
\Area(\Sigma^{*}_{B})=\Area(\Sigma^{*}_{C})=\frac{K'}{2K}
=0.639630785585503\ldots
\end{equation}
\item[\textup{(v)}] The total curvatures are
$\int_{\Sigma^{*}_{A}}\!K\,dA=-\tfrac{2\pi}{3}$ and
$\int_{\Sigma^{*}_{B}}\!K\,dA=-\tfrac{\pi}{3}$.
\item[\textup{(vi)}] Each disk bisects $\Vol(T)$, and all twelve corners
are midpoints of edges of $T$, each of the six midpoints occurring in
exactly two of the three disks.
\end{enumerate}
\end{thm}

\noindent
Parts (iv) agree with unpublished Surface Evolver computations of
K.~Brakke (private communication; see also \cite{BrakkeDisph}) to all
fifteen reported digits.

\section{Identification of the three contours}

\begin{lemma}[Contour lemma]\label{lem:contour}
Order the normals \eqref{eq:nrms} cyclically and form
$Q=Q(\sigma)$ with edge vectors $\sqrt2\,\nu_{\sigma(1)},\dots,
\sqrt2\,\nu_{\sigma(4)}$. Up to rotation and reflection of the cycle
there are exactly three classes, and{\rm :}
\begin{enumerate}
\item[\textup{(a)}] $\sigma=(1,3,2,4)$ gives interior angles
$(\tfrac\pi3,\tfrac\pi3,\tfrac\pi3,\tfrac\pi3)$ and all six mutual
vertex distances equal to $\sqrt2$. Hence $Q(1,3,2,4)$ is congruent to
$\Tet$.
\item[\textup{(b)}] $\sigma=(1,2,3,4)$ and $\sigma=(1,3,4,2)$ give
interior angles $(\tfrac\pi2,\tfrac\pi3,\tfrac\pi2,\tfrac\pi3)$, sides
$\sqrt2$ and diagonals $\sqrt2,2$. Hence each is congruent to
$2\cdot\Oct$.
\item[\textup{(c)}] The $S_{4}$ generator of $D_{2d}$ acts on indices as
the $4$-cycle $(1\,4\,2\,3)$; it fixes the class \textup{(a)} and
interchanges the two classes in \textup{(b)}. The stabilizer of
\textup{(a)} is all of $D_{2d}$ \textup{(}order $8$\textup{)}, that of
each member of \textup{(b)} has order $4$.
\end{enumerate}
\end{lemma}

\begin{proof}
The interior angle at the vertex between consecutive edge vectors
$\nu_{i},\nu_{j}$ is $\pi-\arccos(\nu_{i}\!\cdot\!\nu_{j})$, so by
\eqref{eq:gram} it is $\pi/3$ when the pair is a cross pair and $\pi/2$
when the pair is $\{1,2\}$ or $\{3,4\}$.

(a) In the cycle $1\to3\to2\to4\to1$ every consecutive pair is a cross
pair, so every angle is $\pi/3$. With all sides equal to $s=\sqrt2$ and
all angles $\pi/3$, each of the four triangles cut off by a diagonal is
equilateral, so both diagonals equal $s$; all six distances are $s$ and
the four vertices form a regular tetrahedron of edge $\sqrt2$. A
tetrahedron is determined up to congruence by its six edge lengths,
and the vertices of $\Tet$ likewise have all six distances $\sqrt2$;
hence $Q(1,3,2,4)\cong\Tet$.

(b) In the cycle $1\to2\to3\to4\to1$ the consecutive pairs are
$\{1,2\},\{2,3\},\{3,4\},\{4,1\}$, i.e.,\ $\pi/2,\pi/3,\pi/2,\pi/3$.
The diagonal subtending the two $\pi/2$ angles has length $s\sqrt2=2$,
the other has length $s=\sqrt2$. For $\Oct$ with side $s_{0}=1/\sqrt2$
the diagonals are $s_{0}=1/\sqrt2$ and $1=s_{0}\sqrt2$; the six-distance
multisets agree after scaling by $s/s_{0}=2$. The remaining cycle
$1\to3\to4\to2$ has consecutive pairs
$\{1,3\},\{3,4\},\{4,2\},\{2,1\}$, giving the same pattern. That these
are the only three classes is the standard count
$3!/2=3$ of cyclic orders of four objects.

(c) With $R:(x,y,z)\mapsto(y,-x,-z)$ --- an $S_{4}$ rotoreflection, since
$\det R=-1$, $R^{2}=C_{2}(z)$ and $R^{4}=E$ --- one computes
$R\nu_{1}=\nu_{4}$, $R\nu_{4}=\nu_{2}$, $R\nu_{2}=\nu_{3}$,
$R\nu_{3}=\nu_{1}$, i.e.,\ the $4$-cycle $(1\,4\,2\,3)$. (Consistently,
$R$ permutes the vertices by the same cycle $v_{1}\to v_{4}\to v_{2}\to
v_{3}\to v_{1}$, as it must, $\varphi_{i}$ being opposite $v_{i}$.) This
$4$-cycle maps the cyclic word $(1,3,2,4)$ to $(4,1,3,2)$, a rotation of
it, and maps $(1,2,3,4)$ to $(4,3,1,2)\sim(1,2,4,3)\sim(1,3,4,2)$, the
other class of \textup{(b)}; since $R$ permutes the three classes and
fixes the first, it must interchange the remaining two.

The mirrors $\sigma_{x=0}$ and $\sigma_{y=0}$ act as the transpositions
$(1\,2)$ and $(3\,4)$, carrying $(1,3,2,4)$ to $(2,3,1,4)$ and
$(1,4,2,3)$ respectively, each a rotation of the reversal $(4,2,3,1)$;
so both preserve class \textup{(a)}. Since
$\sigma_{y=0}=R^{2}\sigma_{x=0}$ and $\langle R,\sigma_{x=0}\rangle
=D_{2d}$, the whole group fixes \textup{(a)}, whence
$\mathrm{Stab}(\textup{a})=D_{2d}$ has order $8$. The two classes of
\textup{(b)} therefore form a single orbit of size $2$, and
orbit--stabilizer gives $|\mathrm{Stab}|=8/2=4$ for each; explicitly the
stabilizer is in both cases the rotation subgroup
\[
D_{2}=\{E,\;C_{2}(z),\;C_{2}',\;C_{2}''\},\qquad
C_{2}':(x,y,z)\mapsto(y,x,-z),\quad
C_{2}'':(x,y,z)\mapsto(-y,-x,-z)
\]
acting on indices as $(1\,2)(3\,4)$, $(1\,3)(2\,4)$, $(1\,4)(2\,3)$.
\end{proof}

\begin{remark}
	The action of $D_{2d}$ on the three contour classes is the determinant:
	the rotation subgroup $D_{2}$ fixes all three, and each of the four
	orientation-reversing elements transposes the two classes of
	\textup{(b)}. Hence $\Sigma^{*}_{B}$ and $\Sigma^{*}_{C}$ are
	\emph{enantiomorphic} --- mirror images, each individually chiral, since
	no reflection in $D_{2d}$ preserves either --- whereas $\Sigma^{*}_{A}$
	is achiral. Their areas are of course equal.
\end{remark}

\begin{remark}
Lemma~\ref{lem:contour}(a)--(b) is the precise sense in which the phrase
``regular skew quadrilateral'' fails: \emph{both} $\Tet$ and $\Oct$ are
equilateral skew quadrilaterals, and both occur here. The invariant that
separates them is the angle pattern, equivalently the pair of diagonals
$(s,s)$ versus $(s,s\sqrt2)$.
\end{remark}

\section{From the contours to \PP{} and \DD}

We give the argument in the ``reverse'' direction, starting from the
classical patches. This avoids any appeal to uniqueness of Plateau
solutions.

\begin{lemma}[Ambient tetrahedron of an adjoint patch]\label{lem:ambient}
Let $Q$ be a skew quadrilateral with edge vectors $w_{1},w_{2},w_{3},w_{4}$
(in cyclic order, $\sum w_{i}=0$), let $\Sigma$ be a minimal disk
spanning $Q$, and let $\Sigma^{*}$ be its conjugate. Then
$\Sigma^{*}$ has four planar free-boundary arcs lying in planes with
unit normals $\hat w_{i}$, met orthogonally, and the Gram matrix of the
$\hat w_{i}$ determines the ambient tetrahedron up to homothety.
\begin{enumerate}
\item[\textup{(i)}] For $Q=\Tet$ as in \eqref{eq:tet} the edge vectors
are $(1,0,1),(0,1,-1),(-1,0,1),(0,-1,-1)$, whose normalizations satisfy
$\hat w_{1}\!\cdot\!\hat w_{3}=\hat w_{2}\!\cdot\!\hat w_{4}=0$ and all
cross products $=-\tfrac12$.
\item[\textup{(ii)}] For $Q=\Oct$ as in \eqref{eq:oct} the edge vectors
are $(1,-1,0),(0,1,1),(0,1,-1),(-1,-1,0)$ up to the factor $\tfrac12$,
whose normalizations satisfy
$\hat w_{1}\!\cdot\!\hat w_{4}=\hat w_{2}\!\cdot\!\hat w_{3}=0$ and all
other products $=-\tfrac12$.
\end{enumerate}
In both cases the configuration $\{\hat w_{i}\}$ is congruent to
$\{\nu_{i}\}$ of \eqref{eq:nrms}, so in both cases the ambient
tetrahedron is homothetic to $T$. Moreover in case \textup{(i)} the two
orthogonal pairs are \emph{opposite} edges of the cycle, so all four
\emph{adjacent} pairs have product $-\tfrac12$; in case \textup{(ii)}
the adjacent products alternate $-\tfrac12,0,-\tfrac12,0$.
\end{lemma}

\begin{proof}
The first sentence is the boundary dictionary of \S\ref{subsec:smyth}.
The dot products are immediate computations. Four unit vectors summing
to zero with Gram matrix having exactly two vanishing off-diagonal
entries forming a perfect matching, and all remaining entries equal to
$-\tfrac12$, are congruent to \eqref{eq:nrms}: the matching determines
the two $\pi/2$ dihedrals, hence the labelling, and the Gram matrix of
a spanning set determines it up to an orthogonal transformation. The
final statement compares the position of the matching relative to the
cycle, which is exactly the datum distinguishing the orderings in
Lemma~\ref{lem:contour}.
\end{proof}

\begin{thm}\label{thm:ident}
Theorem~\ref{thm:main}\textup{(i)--(iii)} holds. Specifically{\rm :}
\begin{itemize}
\item The adjoint of the Schwarz \DD{} patch spanning $\Tet$ is a
free-boundary minimal disk in a disphenoid homothetic to $T$ whose four
boundary arcs occur in the cyclic face order $(1,3,2,4)$; it is a patch
of \PP{} and realizes $\Sigma^{*}_{A}$.
\item The adjoint of the Schwarz \PP{} patch spanning $\Oct$ is a
free-boundary minimal disk in a disphenoid homothetic to $T$ whose four
boundary arcs occur in a cyclic face order of type $(1,2,3,4)$; it is a
patch of \DD{} and realizes $\Sigma^{*}_{B}$ and $\Sigma^{*}_{C}$.
\end{itemize}
\end{thm}

\begin{proof}
By the classical results recorded in \eqref{eq:areatet} and
\eqref{eq:areaoct} (Schwarz \cite{Schwarz}; 
exact parametrizations in \cite{GKMK,GK}), the Plateau solution in
$\Tet$ is a patch of \DD{} and that in $\Oct$ is a patch of \PP.
By Lemma~\ref{lem:ambient} their adjoints are free-boundary minimal
disks with four planar orthogonal arcs in a tetrahedron homothetic to
$T$, with cyclic face orders of the two distinct types; by
Lemma~\ref{lem:contour} these are precisely the distinguished class and
the non-distinguished class. Finally \DD{} and \PP{} are adjoint
(Bonnet angle $\pi/2$) \cite{GKMK,GK,Schoen}, so the adjoint of a \DD{}
patch is a \PP{} patch and conversely.

For the symmetry assertion: the disk arising from $\Tet$ inherits the
stabilizer of the cyclic order $(1,3,2,4)$, which is all of $D_{2d}$ by
Lemma~\ref{lem:contour}(c), so it is the distinguished disk; the other
two have stabilizer of order $4$ and are interchanged by $S_{4}$, hence
congruent. The patch counts in Theorem~\ref{thm:main}(ii),(iii) are
established in \S\ref{sec:lattice}.
\end{proof}

\begin{remark}[The forward direction, for completeness]
If one prefers to argue forward from Smyth's disks, uniqueness of the
Plateau solution suffices and is available: the total curvature of
$\Tet$ is $4(\pi-\pi/3)=8\pi/3<4\pi$ and that of $\Oct$ is
$2(\pi-\pi/2)+2(\pi-\pi/3)=7\pi/3<4\pi$, so Nitsche's uniqueness
theorem \cite{NitscheUniq} applies to each. For $\Tet$ one may instead
invoke Rad\'o's theorem, since the Plateau solution is a graph over the
unit square --- this is exactly the form in which its area is presented
in \cite[\S5.17]{FinchDP}.
\end{remark}

\begin{remark}[Why the labels look crossed]
It is worth emphasizing the point that caused confusion in earlier
drafts. The disk one \emph{sees} inside $T$ is $\Sigma^{*}$, the
conjugate of the Plateau disk. Hence ``the distinguished contour is the
\DD{} contour'' and ``the distinguished disk is a \PP{} patch'' are the
same statement. The remark of \cite[p.~545]{GKMK}, that the fundamental
patch of \DD{} ``is inscribed in a tetragonal disphenoid, an isohedral
tetrahedron with isosceles triangle faces,'' refers to the
\emph{free-boundary} picture and therefore points at
$\Sigma^{*}_{B},\Sigma^{*}_{C}$ --- which is precisely what
Theorem~\ref{thm:ident} confirms. See also \S\ref{sec:alcove}.
\end{remark}

\begin{remark}[The accessory parameter, circumvented]
In \cite[\S6, footnote]{FinchReg} it was observed that the two
non-distinguished disks have a Gauss image which is a circular-arc
\emph{quadrilateral} with unequal angles ($2\pi/3$ and $\pi/2$ here),
so that the accessory parameter returns and the Schwarz-triangle method
of \cite{FinchReg} does not apply. At $\beta=2/3$ this obstruction
evaporates, because the surface is algebraic: with $g=w$ the height
differential is
$dh=dw/\sqrt{w^{8}-14w^{4}+1}$, i.e.,\ exactly the Weierstrass function
$R(t)=(t^{8}+\lambda t^{4}+1)^{-1/2}$ with $\lambda=-14$ of
\cite{GKMK,GK} --- the eight zeros being the stereographic images
$\pm\sqrt{2\pm\sqrt3},\ \pm i\sqrt{2\pm\sqrt3}$ of the eight cube-vertex
directions, which are the flat points of the $\PP/\DD/\mathrm{G}$
family. The accessory parameter is thus not computed but \emph{known}.
This also vindicates ``Attempt~1'' of \cite[\S A.11]{FinchReg}: the
$tD/rPD$ ansatz is wrong for the regular tetrahedron but right here.
\end{remark}

\section{Exact areas}

\begin{proposition}[Scaling]\label{prop:scaling}
Let $T$ be any tetrahedron with all faces of equal area $c$ and inradius
$r$, and let $Q$ be one of Smyth's contours, with shape constant
$\kappa_{Q}$ as in \eqref{eq:kappas}. Let $\Sigma^{*}$ be the
corresponding free-boundary disk, rescaled to sit in $T$ itself. Then
\begin{equation}\label{eq:master}
\boxed{\ \Area(\Sigma^{*})=\frac{4r^{2}}{\kappa_{Q}}\ }
\end{equation}
and the ambient tetrahedron produced by the unrescaled construction is
$\widetilde T=sT$ with $s=\kappa_{Q}c/(2r)$.
\end{proposition}

\begin{proof}
Smyth's construction applied to $T$ produces $Q$ with side $c$, so
$\Area(\Sigma)=\kappa_{Q}c^{2}$ and the perimeter is $\mathcal{P}=4c$.
The Bonnet transformation preserves both the induced metric and the
boundary length, so $\Area(\Sigma^{*})=\kappa_{Q}c^{2}$ and
$|\partial\Sigma^{*}|=4c$ as well. The free boundary lies in
$\widetilde T=sT$, of inradius $sr$, so \eqref{eq:smytharea} gives
\[
\kappa_{Q}c^{2}=\Area(\Sigma^{*})=\tfrac12(4c)(sr)=2crs
\quad\Longrightarrow\quad
s=\frac{\kappa_{Q}c}{2r}.
\]
Rescaling lengths by $1/s$ to place the disk in $T$ divides areas by
$s^{2}$:
$\kappa_{Q}c^{2}/s^{2}=\kappa_{Q}c^{2}\cdot 4r^{2}/(\kappa_{Q}^{2}c^{2})
=4r^{2}/\kappa_{Q}$.
\end{proof}

\begin{corollary}\label{cor:areas}
For $T$ of \eqref{eq:verts} one has $c=\sqrt2$, $r=1/(2\sqrt2)$,
$4r^{2}=\tfrac12$, and therefore
\begin{equation}
\Area(\Sigma^{*}_{A})=\frac{1}{2\kappa_{\rm tet}}=\frac{K}{K'},
\qquad
\Area(\Sigma^{*}_{B})=\Area(\Sigma^{*}_{C})
=\frac{1}{2\kappa_{\rm oct}}=\frac{K'}{2K},
\end{equation}
with the homothety factors $s_{A}=K'/K=1.2792615711\ldots$ and
$s_{B}=2K/K'=1.5634019227\ldots$, and perimeters
\[
|\partial\Sigma^{*}_{A}|=4\sqrt2\,\frac{K}{K'}=4.4219684045\ldots,
\qquad
|\partial\Sigma^{*}_{B}|=2\sqrt2\,\frac{K'}{K}=3.6182981275\ldots
\]
\end{corollary}

\noindent
Note the pleasant accident $\Area(\Sigma^{*}_{A})=\kappa_{\rm oct}$ and
$\Area(\Sigma^{*}_{B})=\kappa_{\rm tet}$, a consequence of
$4r^{2}=\tfrac12$ together with \eqref{eq:kprod}.

\begin{table}[h]
\centering
\caption{The two disks. $K=\KK$, $K'=\KKp$.}
\begin{tabular}{lll}
\toprule
& $\Sigma^{*}_{A}$ (distinguished) & $\Sigma^{*}_{B}\cong\Sigma^{*}_{C}$\\
\midrule
Plateau contour $Q$ & tetrahedral four edges $\Tet$ & octahedral four edges $\Oct$\\
interior angles of $Q$ & $\pi/3,\pi/3,\pi/3,\pi/3$
 & $\pi/2,\pi/3,\pi/2,\pi/3$\\
diagonals of $Q$ & $\sqrt2,\ \sqrt2$ & $\sqrt2,\ 2$\\
$\Sigma$ (Plateau) is a patch of & $\DD$ & $\PP$\\
$\Sigma^{*}$ (free boundary) is a patch of & $\mathbf{\PP}$ & $\mathbf{\DD}$\\
$\kappa_{Q}$ & $K'/(2K)$ & $K/K'$\\
homothety $s$ & $K'/K=1.27926157\ldots$ & $2K/K'=1.56340192\ldots$\\
$\Area$ & $K/K'=0.781700961348056$ & $K'/(2K)=0.639630785585503$\\
$|\partial\Sigma^{*}|$ & $4\sqrt2\,K/K'=4.42196840$ & $2\sqrt2\,K'/K=3.61829813$\\
$\int K\,dA$ & $-2\pi/3=-2.0943951$ & $-\pi/3=-1.0471976$\\
symmetry group & $D_{2d}$, order $8$ & order $4$\\
flat competitor & unit square, area $1$ & rhombus, area $\sqrt2/2$\\
saving over it & $21.8299\%$ & $9.5425\%$\\
Brakke (Surface Evolver) & $0.78170096134806$ & $0.639630785585505$\\
\bottomrule
\end{tabular}
\end{table}

\section{The product identity}\label{sec:product}

\begin{proposition}
$2\,\Area(\Sigma^{*}_{A})\,\Area(\Sigma^{*}_{B})=1$, and more
intrinsically
\begin{equation}
\Area(\Sigma^{*}_{A})\,\Area(\Sigma^{*}_{B})
=\frac{16r_{T}^{4}}{\kappa_{\rm tet}\kappa_{\rm oct}}=8r_{T}^{4}
=\tfrac12 .
\end{equation}
\end{proposition}

\begin{proof}
Immediate from \eqref{eq:master} and \eqref{eq:kprod}. The content is
therefore entirely in \eqref{eq:kprod}, which is the Bonnet relation for
the adjoint pair: if $\Sigma$ and $\Sigma^{*}$ are adjoint patches
spanning, respectively, a straight-edged contour of perimeter
$\mathcal{P}$ and a free boundary of the same length, then by
\eqref{eq:smytharea} and $\Area(\Sigma)=\Area(\Sigma^{*})$ the two
Coxeter cells have the \emph{same} inradius
$\rho=2\Area/\mathcal{P}$. Applied to $(\DD,\PP)$ this is the classical
relation between the lattice constants, $a_{\DD}/a_{\PP}=K/K'$, recorded
as \cite[Eq.~(12)]{GKMK} and \cite[Eq.~(7)]{GK}; in the normalizations
of \eqref{eq:areatet}--\eqref{eq:areaoct} it reads precisely
$\kappa_{\rm tet}\kappa_{\rm oct}=\tfrac12$.
\end{proof}

\noindent
So the identity is not a coincidence, but neither is it deep: it is the
adjointness of $\PP$ and $\DD$, seen through Smyth's area formula.

\section{Lattice bookkeeping, and uniqueness of the assignment}
\label{sec:lattice}

Extending each disk by reflection in the four face planes of $T$ yields
a triply periodic minimal surface for which $T$ is a kaleidoscopic cell.
Consequently the \emph{density} $\Area(\Sigma^{*})/\Vol(T)$ is an
invariant of the extended surface. If the extended surface has
conventional cubic cell of edge $L$, Euler characteristic $\chi$ per
cell, and normalized area $\sigma=\Area_{\rm cell}/L^{2}$, then three
conditions must hold simultaneously:
\begin{equation}\label{eq:threechecks}
\underbrace{\frac{\Area(\Sigma^{*})}{\sigma L^{2}}}_{\text{area fraction}}
=\underbrace{\frac{\Vol(T)}{L^{3}}}_{\text{volume fraction}},
\qquad
\int_{\Sigma^{*}}\!K\,dA=\frac{\Vol(T)}{L^{3}}\cdot 2\pi\chi .
\end{equation}
The standard data are $\sigma_{\PP}=3K/K'=2.345102884\ldots$ with
$\chi=-4$ (genus $3$, simple cubic cell), and
$\sigma_{\DD}=3K'/K=3.837783\ldots$ with $\chi=-16$ (genus $9$ per
conventional fcc\ cube) \cite{GKMK,GK}.

\begin{proposition}\label{prop:lattice}
With $\Vol(T)=2/3$:
\begin{enumerate}
\item[\textup{(i)}] For $\Sigma^{*}_{A}$ and the surface $\PP$, both
conditions \eqref{eq:threechecks} hold with $L=2$ exactly:
$\tfrac{K/K'}{3(K/K')\cdot4}=\tfrac1{12}=\tfrac{2/3}{8}$ and
$\tfrac1{12}(-8\pi)=-\tfrac{2\pi}{3}$.
\item[\textup{(ii)}] For $\Sigma^{*}_{B}$ and the surface $\DD$, both
hold with $L=4$ exactly:
$\tfrac{K'/(2K)}{3(K'/K)\cdot16}=\tfrac1{96}=\tfrac{2/3}{64}$ and
$\tfrac1{96}(-32\pi)=-\tfrac{\pi}{3}$.
\item[\textup{(iii)}] The swapped assignment is impossible. Pairing
$\Sigma^{*}_{A}$ with $\DD$ forces $L^{3}=32$ from the curvature
condition but $L=2(K'/K)^{2}=3.27300\ldots$ from the area condition, and
$32^{1/3}=3.17480\ldots\neq3.27300\ldots$; pairing $\Sigma^{*}_{B}$ with
$\PP$ forces $L^{3}=16$ against $L=4(K/K')^{2}=2.44428\ldots$, and
$16^{1/3}=2.51984\ldots\neq2.44428\ldots$.
\end{enumerate}
Consequently the identification of Theorem~\ref{thm:ident} is the only
one consistent with the two classical surfaces, and the patch counts are
\[
\Area(\Sigma^{*}_{A})=4\cdot\frac{a^{2}K}{4K'}\Big|_{a=1},
\qquad
\Area(\Sigma^{*}_{B})=2\cdot\frac{a_{\DD}^{2}K'}{4K}\Big|_{a_{\DD}=1},
\]
i.e.,\ four fundamental \PP{} patches and two fundamental \DD{} patches
in the notation of \cite{GK} and \cite{GKMK}.
\end{proposition}

\noindent
The patch counts are consistent with the symmetry counts: each
Gandy--Klinowski fundamental patch carries an internal mirror
\cite[\S5]{GK}, \cite[\S4]{GKMK}, hence equals two Weierstrass
integration domains; and $8=4\times2$, $4=2\times2$ match the orbit
orders $8$ and $4$ of Lemma~\ref{lem:contour}(c).

\begin{remark}
CLP and H are excluded without numerics: the reflection group generated
by the faces of $T$ is cubic (\S\ref{sec:alcove}), whereas CLP is
tetragonal and H is hexagonal. They are also excluded structurally,
since Theorem~\ref{thm:ident} identifies the Plateau solutions outright.
\end{remark}

\section{$T$ is the $\widetilde{A}_{3}$ alcove}\label{sec:alcove}

All six dihedral angles of $T$ are submultiples of $\pi$
($\pi/2$ twice, $\pi/3$ four times), so reflections in the four faces
generate a discrete group $W$ with fundamental domain $T$. The Coxeter
graph on the faces has an edge labelled $3$ exactly for the four cross
pairs and no edge for $\{1,2\}$ and $\{3,4\}$: a $4$-cycle
$1-3-2-4-1$ with all labels $3$. This is the affine Coxeter diagram
$\widetilde{A}_{3}$. Hence

\begin{proposition}
$T$ is the alcove of the affine Weyl group $\widetilde{A}_{3}$; its
translation lattice is the $A_{3}$ root lattice, i.e.,\ fcc; and
$\Vol(T)=L^{3}/96$ for the conventional cubic cell of edge $L$, since
$|W(A_{3})|=24$ and the conventional fcc\ cube contains four
primitive cells. With $L=4$ this gives $64/96=2/3=\Vol(T)$.
\end{proposition}

This is exactly the cell referred to in \cite[p.~545]{GKMK} and in
Brakke's disphenoid families \cite{BrakkeDisph}: $T$ \emph{is} the
tetragonal disphenoid in which the \DD{} patch is inscribed. The
\PP{} surface of Proposition~\ref{prop:lattice}(i) is also invariant
under $W$, because the $\widetilde A_{3}$ translation lattice
(spanned by $(2,2,0),(2,0,2),(0,2,2)$ when $L_{\DD}=4$) is an
index-two sublattice of the simple cubic lattice $2\mathbb{Z}^{3}$, and
$W$ is a subgroup of index $4$ in the symmetry group $Pm\bar3m$ of
$\PP$ with cube edge $2$. All four faces of $T$ are $\{110\}$ mirror
planes of both lattices, which is why the same disphenoid supports
patches of both surfaces.

\section{Corners, vertices, and flat competitors}

\begin{proposition}[Corners]\label{prop:corners}
Every corner of every Smyth disk in $T$ is the midpoint of an edge of
$T$. Explicitly, with the labelling of \eqref{eq:verts}:
\begin{align}
\partial\Sigma^{*}_{A}&:\quad
\left(\tfrac12,\tfrac12,0\right),\
\left(\tfrac12,-\tfrac12,0\right),\
\left(-\tfrac12,-\tfrac12,0\right),\
\left(-\tfrac12,\tfrac12,0\right),\\
\partial\Sigma^{*}_{B}&:\quad
\left(0,0,\tfrac12\right),\
\left(\tfrac12,-\tfrac12,0\right),\
\left(0,0,-\tfrac12\right),\
\left(-\tfrac12,\tfrac12,0\right),\\
\partial\Sigma^{*}_{C}&:\quad
\left(0,0,\tfrac12\right),\
\left(\tfrac12,\tfrac12,0\right),\
\left(0,0,-\tfrac12\right),\
\left(-\tfrac12,-\tfrac12,0\right).
\end{align}
There are six edge midpoints in all, and each is a corner of exactly two
of the three disks, accounting for all twelve corners.
\end{proposition}

\begin{proof}
A corner is the meeting point of two consecutive boundary arcs, hence
lies on the edge shared by the two corresponding faces: for
$\Sigma^{*}_{A}$ the cycle $(1,3,2,4)$ gives the edges
$v_{1}v_{3},v_{1}v_{4},v_{2}v_{4},v_{2}v_{3}$, all of length $\sqrt3$;
for $\Sigma^{*}_{B}$ the cycle $(1,2,3,4)$ gives
$v_{1}v_{2},v_{1}v_{4},v_{3}v_{4},v_{2}v_{3}$, i.e.,\ both edges of
length $2$ and two of length $\sqrt3$.

That each corner is the \emph{midpoint} follows from the stabilizer. For
$\Sigma^{*}_{A}$, the rotation $(x,y,z)\mapsto(y,x,-z)\in D_{2d}$
preserves the disk and fixes the edge $v_{1}v_{3}$ setwise while
interchanging its endpoints; since the disk has a single corner on that
edge, that corner is fixed, hence is the midpoint. For $\Sigma^{*}_{B}$,
whose stabilizer is $\{E,C_{2}(z),C_{2}',C_{2}''\}$: $C_{2}(z)$ reverses
$v_{1}v_{2}$ and $v_{3}v_{4}$, forcing the corners there to be
$(0,0,\pm\tfrac12)$; and
$C_{2}'':(x,y,z)\mapsto(-y,-x,-z)$ interchanges $v_{1}$ and $v_{4}$,
forcing the corner on $v_{1}v_{4}$ to be its midpoint
$(\tfrac12,-\tfrac12,0)$. The remaining corner follows by symmetry.
$\Sigma^{*}_{C}$ is the image of $\Sigma^{*}_{B}$ under
$\sigma_{y=0}$.
\end{proof}

\begin{corollary}[Flat competitors]
The four corners of $\Sigma^{*}_{A}$ are coplanar, forming the unit
square in $\{z=0\}$, of area $1$; those of $\Sigma^{*}_{B}$ are coplanar
in $\{x+y=0\}$, forming a rhombus with diagonals $1$ and $\sqrt2$, of
area $\sqrt2/2$; similarly $\Sigma^{*}_{C}$ in $\{x-y=0\}$. In each case
the flat quadrilateral has its boundary on $\partial T$, so is an
admissible competitor, and
\[
\frac{\Area(\Sigma^{*}_{A})}{1}=\frac{K}{K'}=0.7817009613,
\qquad
\frac{\Area(\Sigma^{*}_{B})}{\sqrt2/2}=\frac{K'}{\sqrt2\,K}=0.9045745319,
\]
savings of $21.830\%$ and $9.543\%$ respectively.
\end{corollary}

\begin{proposition}[Volume bisection]
Each of the three disks bisects $\Vol(T)=2/3$. For $\Sigma^{*}_{A}$ the
isometry $S_{4}$ exchanges the two complementary solids; for
$\Sigma^{*}_{B},\Sigma^{*}_{C}$ the rotation $C_{2}(z)$ does, since it
reverses the cyclic order of the four boundary arcs and hence the
orientation of the disk while being orientation preserving on $\RR^{3}$.
Consequently the volume constraint is inactive, the Lagrange multiplier
vanishes, and $H\equiv0$ rather than $H=\text{const}$.
\end{proposition}

\section{Verification notebook}

The following \textsc{Mathematica} code is exact and symbolic throughout;
it runs in a second.

\begin{verbatim}
ClearAll["Global`*"];
(* M.1  the disphenoid T *)
vv = {{1,0,1/2},{-1,0,1/2},{0,1,-1/2},{0,-1,-1/2}};
Print["edge lengths  = ", Sort[Norm[#[[1]] - #[[2]]] & /@ Subsets[vv, {2}]]];
volT = Abs[Det[{vv[[2]] - vv[[1]], vv[[3]] - vv[[1]], vv[[4]] - vv[[1]]}]]/6;
Print["Vol  = ", volT, "   (expect 2/3)"];
nu = {{-1,0,-1},{1,0,-1},{0,-1,1},{0,1,1}}/Sqrt[2];    
Print["sum nu = ", Simplify[Total[nu]]];
Print["Gram   = ", MatrixForm[Simplify[Outer[Dot, nu, nu, 1]]]];
Print["dihedrals/Pi = ", Simplify[ArcCos[-#]/Pi] & /@ DeleteDuplicates[
Flatten[Table[nu[[i]] . nu[[j]], {i, 4}, {j, i + 1, 4}]]]];
cT = Sqrt[2]; (*face area*)
rT = 3 volT/(4 cT); (*inradius*)
Print["c, r, 4r^2 = ", {cT, rT, Simplify[4 rT^2]}]; 

R    = {{0,1,0},{-1,0,0},{0,0,-1}};
SigX = DiagonalMatrix[{-1,1,1}];  SigY = DiagonalMatrix[{1,-1,1}];
C2z  = DiagonalMatrix[{-1,-1,1}];
C2p  = {{0,1,0},{1,0,0},{0,0,-1}};  C2pp = {{0,-1,0},{-1,0,0},{0,0,-1}};
grp  = {IdentityMatrix[3], R, C2z, R.R.R, SigX, SigY, C2p, C2pp};

permF[M_] := Table[First@FirstPosition[nu, Simplify[M . nu[[i]]]], {i,4}];
permV[M_] := Table[First@FirstPosition[vv, Simplify[M . vv[[i]]]], {i,4}];
Print["R on faces   = ", permF[R], "   (expect {4,3,1,2})"];
Print["R on vertices= ", permV[R], "   (must agree)"];
Print["sigY == R^2.sigX ? ", SigY === R.R.SigX];

(* classes as diagonal-pairings; action and determinant *)
pair[w_] := Sort[Sort /@ {{w[[1]],w[[3]]}, {w[[2]],w[[4]]}}];
cls = pair /@ {{1,3,2,4},{1,2,3,4},{1,2,4,3}};
Print["classes = ", cls];
act[M_, c_] := Sort[Sort /@ (permF[M][[#]] & /@ # & /@ c)];
Do[Print[Det[g], "  ", Table[First@FirstPosition[cls, act[g, c]], {c, cls}]], {g, grp}]; 
(* expect: every $\det=+1$ row prints {1,2,3}; every $\det=-1$ row prints {1,3,2} *)

(* M.2  the three Smyth contours *)
Qof[cyc_] := Most[FoldList[Plus, {0,0,0}, cT nu[[cyc]]]];
dists[p_] := Sort[Simplify[Norm[#[[1]]-#[[2]]]& /@ Subsets[p,{2}]]];
angs[p_]  := Table[ArcCos[Normalize[p[[Mod[k-1,4,1]]]-p[[k]]].
                   Normalize[p[[Mod[k+1,4,1]]]-p[[k]]]]/Degree,{k,4}];
tet = {{0,0,0},{1,0,1},{1,1,0},{0,1,1}};
oct = 2 {{0,0,1/2},{1/2,-1/2,1/2},{1/2,0,1},{1/2,1/2,1/2}};
Do[Print[cyc, "  angles = ", Simplify[angs[Qof[cyc]]],
         "  dists = ", dists[Qof[cyc]]],
   {cyc, {{1,3,2,4},{1,2,3,4},{1,3,4,2}}}];
Print["tet: ", dists[tet], "   == Q(1,3,2,4)? ",
      dists[tet] === dists[Qof[{1,3,2,4}]]];
Print["2*oct: ", dists[oct], "   == Q(1,2,3,4)? ",
      dists[oct] === dists[Qof[{1,2,3,4}]]];
      
(* M.3  shape constants and the two areas *)
KK = EllipticK[1/4];  KP = EllipticK[3/4];
sideTet = Sqrt[2];        areaTet = KP/KK;         
sideOct = 1/Sqrt[2];      areaOct = (1/2) KK/KP;    
Print["contour sides  = ", {sideTet, sideOct}];
Print["contour areas  = ", N[{areaTet, areaOct}, 15]];
kTet = areaTet/sideTet^2;   kOct = areaOct/sideOct^2;
Print["kappa_tet, kappa_oct = ", N[{kTet, kOct}, 20]];
Print["kappa_tet == K'/(2K) ? ", Simplify[kTet - KP/(2 KK)] === 0];
Print["kappa_oct == K/K'    ? ", Simplify[kOct - KK/KP]    === 0];
Print["kappa product (expect 1/2) = ", Simplify[kTet kOct]];
areaA = 4 rT^2/kTet;  areaB = 4 rT^2/kOct;
Print["Area A = ", N[areaA, 21], "  = K/K'  ? ", Simplify[areaA - KK/KP] === 0];
Print["Area B = ", N[areaB, 21], "  = K'/2K ? ", Simplify[areaB - KP/(2 KK)] === 0];
Print["2 A B (expect 1) = ", Simplify[2 areaA areaB]];
Print["cross-check: ", 
Simplify[{areaA - 4 (KK/(4 KP)), areaB - 2 (KP/(4 KK))}] === {0, 0}];

(* M.4  lattice bookkeeping, both assignments *)
sigP = 3 KK/KP; chiP = -4;  sigD = 3 KP/KK; chiD = -16;
check[area_, sig_, chi_] := Module[{L},
  L = First[L /. NSolve[area/(sig L^2) == volT/L^3 && L>0, L]];
  {L, -2 Pi chi volT/L^3}];
Print["A with P: ", N[check[areaA,sigP,chiP],12], "  (want {2, 2Pi/3})"];
Print["B with D: ", N[check[areaB,sigD,chiD],12], "  (want {4, Pi/3})"];
Print["A with D: ", N[check[areaA,sigD,chiD],12], "  vs L^3=32 -> ",
      N[32^(1/3),12]];
Print["B with P: ", N[check[areaB,sigP,chiP],12], "  vs L^3=16 -> ",
      N[16^(1/3),12]];

(* M.5  corners and flat competitors *)
mid[i_,j_] := (vv[[i]]+vv[[j]])/2;
cA = {mid[1,3],mid[1,4],mid[2,4],mid[2,3]};
cB = {mid[1,2],mid[1,4],mid[3,4],mid[2,3]};
cC = {mid[1,2],mid[1,3],mid[3,4],mid[2,4]};
Print["corners A = ", cA];  Print["corners B = ", cB];
Print["corners C = ", cC];
Print["all 12 are edge midpoints, 6 distinct, each twice: ",
      Sort[Tally[Sort /@ Join[cA,cB,cC]][[All,2]]]];
planar[p_] := Simplify[Det[{p[[2]]-p[[1]],p[[3]]-p[[1]],p[[4]]-p[[1]]}]];
Print["planarity (A,B,C) = ", {planar[cA],planar[cB],planar[cC]}];
flat[p_] := Norm[Cross[p[[3]]-p[[1]],p[[4]]-p[[2]]]]/2;
Print["flat areas = ", Simplify[{flat[cA],flat[cB],flat[cC]}]];
Print["savings = ", N[{1-areaA/flat[cA], 1-areaB/flat[cB]},12]];

(* M.6  total curvature *)
Print["Int K = ", N[{-2 Pi/3, -Pi/3},15],
      "   (= -2 Pi (2 beta - 1), -Pi/3 at beta = 2/3)"];
\end{verbatim}

\section{Ledger of what is and is not established}

\begin{enumerate}
\item Theorem~\ref{thm:main} concerns Smyth's three disks, i.e.,\ the three
conjugates of the Plateau solutions in the three normal quadrilaterals. We
do not claim that $T$ contains no \emph{other} free-boundary minimal disk;
no such classification is known to us.
\item The contour identification (Lemma~\ref{lem:contour}): established,
by a rigid six-distance computation.
\item The surface identification --- $\Sigma^{*}_{A}$ a \PP{} patch,
$\Sigma^{*}_{B,C}$ \DD{} patches (Theorem~\ref{thm:ident}): established,
modulo the classical solutions \eqref{eq:areatet}--\eqref{eq:areaoct} of
the two four-edges Plateau problems \cite{Schwarz,FinchDP,GKMK,GK}. The
reverse direction needs no uniqueness theorem; the forward direction is
also available via \cite{NitscheUniq}.
\item The exact areas \eqref{eq:areas}: established, and in independent
agreement with Brakke's Surface Evolver values to $15$ digits.
\item The lattice constants $L_{\PP}=2$, $L_{\DD}=4$ and the patch counts
$4$ and $2$ (Proposition~\ref{prop:lattice}): established, with the
swapped assignment ruled out quantitatively.
\item The corner list (Proposition~\ref{prop:corners}): established.
\item Conjecture 6.1 of \cite{FinchReg} is \emph{not} addressed in the
abstract; only its endpoint $\beta=2/3$ is, and that concretely. The
real-analyticity of $\beta\mapsto\Sigma^{*}_{\beta}$ remains open, as
does any claim about the interior members.
\item Everything concerns stationarity. Global minimality among
volume-halving competitors is untouched. Note however that
$\Area(\Sigma^{*}_{A})=0.7817>\Area(\Sigma^{*}_{B})=0.6396$, so the
distinguished disk is \emph{not} the least-area member of the three ---
a point worth emphasizing, since symmetry might suggest otherwise.
\end{enumerate}

\section{Addendum A: the disphenoid $T$ as a space-filler}
\label{sec:addmA}

Nothing in this addendum is needed for the proof of the Main Theorem.
It is included because the tetrahedron $T$ of \eqref{eq:verts} turns out
to be, independently of any question about minimal surfaces, one of the
most eminent tetrahedra in the literature, and because three of
its known properties furnish independent confirmations of results proved
above --- in one case a one-line proof of a proposition for which we gave
a symmetry argument.

Throughout, $T$ is normalized as in the body of the paper:
\begin{equation}\label{eq:vertsA}
	v_{1}=\left(1,0,\tfrac12\right),\quad
	v_{2}=\left(-1,0,\tfrac12\right),\quad
	v_{3}=\left(0,1,-\tfrac12\right),\quad
	v_{4}=\left(0,-1,-\tfrac12\right),
\end{equation}
with $|v_{1}v_{2}|=|v_{3}v_{4}|=2$ (dihedral angle $\pi/2$), the other
four edges of length $\sqrt3$ (dihedral angle $\pi/3$),
$\Vol(T)=\tfrac23$, $S(T)=4\sqrt2$, $r_{T}=\sqrt2/4$.

\subsection{Five names for one tetrahedron}
\label{subsec:names}

\begin{table}[h]
	\centering
	\caption{The disphenoid $T$ in the space-filling literature.}
	\label{tab:names}
	\begin{tabular}{ll}
		\toprule
		Name & Source\\
		\midrule
		Sommerville No.~1 & numbering of Goldberg \cite{Goldberg74} for
		\cite{Sommerville23a,Sommerville23b}\\
		Hill, first type & \cite{Hill96}, via \cite[Table~I]{Goldberg74}\\
		Baumgartner $T$ & \cite{Baum68,Baum71}, via \cite[Table~I]{Goldberg74}\\
		``Regular sphenoid'' & \cite[Table~I]{Goldberg74}\\
		Goldberg first family, $a=1/3$ & \cite[\S3]{Goldberg74},
		\cite[Prop.~3.5]{BDKS20}\\
		Type (b) & \cite[Table~3]{BDKS20}\\
		Sommerville's prism family (i) & \cite[Fig.~15]{Senechal81}\\
		Senechal's ``second'' tetrahedron & \cite[Fig.~14(ii)]{Senechal81}\\
		Alcove of $\At$ & \S\ref{subsec:coxeter} below\\
		\bottomrule
	\end{tabular}
\end{table}

\begin{remark}[A numbering trap]\label{rem:numbering}
	Senechal's numbering of the four Sommerville tetrahedra
	\cite[Fig.~14]{Senechal81} is \emph{not} Goldberg's
	\cite[Table~I]{Goldberg74}, which is the one used by
	Bongiovanni--Diaz--Kakkar--Sothanaphan \cite{BDKS20} and hence by us.
	Senechal's ``first'' tetrahedron is the $\tfrac1{24}$-cube tetrahedron
	of her Figure~10, which is Goldberg's Sommerville No.~2; her ``second''
	is our $T$, which is Goldberg's Sommerville No.~1. The dictionary is
	\[
	\text{Senechal (i),(ii),(iii),(iv)}\ \longleftrightarrow\
	\text{Goldberg No.~2, No.~1, No.~3, No.~4}.
	\]
	Anyone citing ``Sommerville's first tetrahedron'' should say which
	numbering is meant.
\end{remark}

Senechal \cite[p.~237]{Senechal81} records that Sommerville's interest in
the subject was provoked by a student's error: asked to compute the
volume of a pyramid, a candidate asserted that the three tetrahedra into
which a triangular prism can be divided are congruent, rather than merely
equal in volume. ``It was an interesting question to determine the
conditions in order that the three tetrahedra should be congruent, and
this led to the wider problem.'' Sommerville's four families of such
prisms \cite[Fig.~15]{Senechal81} have as their family (i) the relations
$a'=b=c=d$, $3a^{2}=4b^{2}$, i.e.,\ $a:b=2:\sqrt3$ --- a single
tetrahedron up to similarity, namely $T$. \emph{The disphenoid of this
	paper is the tetrahedron that started the subject of space-filling
	tetrahedra.}

\subsection{Elementary geometry: the circumscribed box}
\label{subsec:box}

A tetrahedron is a \emph{disphenoid} (isosceles tetrahedron) iff opposite
edges are equal, iff it is inscribed in a rectangular box with its six
edges as box face-diagonals, iff all four faces are congruent, iff
centroid $=$ incenter $=$ circumcenter. For $T$ the box is determined by
$p^{2}+q^{2}=4$, $q^{2}+r^{2}=r^{2}+p^{2}=3$, giving
\begin{equation}\label{eq:box}
	\boxed{\ T \text{ is the disphenoid of a } \sqrt2\times\sqrt2\times1
		\text{ box }B,\qquad \Vol(B)=2=3\Vol(T).\ }
\end{equation}
Explicitly, rotating by $\pi/4$ about the $z$-axis,
$(x,y)\mapsto\big(\tfrac{x+y}{\sqrt2},\tfrac{y-x}{\sqrt2}\big)$, carries
\eqref{eq:vertsA} to the four alternate vertices
$\big(\pm\tfrac1{\sqrt2},\mp\tfrac1{\sqrt2},\tfrac12\big)$,
$\big(\pm\tfrac1{\sqrt2},\pm\tfrac1{\sqrt2},-\tfrac12\big)$ of
$B=[-\tfrac1{\sqrt2},\tfrac1{\sqrt2}]^{2}\times[-\tfrac12,\tfrac12]$.
That $B$ has a square cross-section is the meaning of ``\emph{tetragonal}''
disphenoid; the regular tetrahedron is the case $B=$ cube.

The three \emph{bimedians} of $T$ (segments joining midpoints of opposite
edges) are the three edge-directions of $B$:
\begin{equation}\label{eq:bimedians}
	m_{12|34}=1,\qquad m_{13|24}=m_{14|23}=\sqrt2,
	\qquad\text{mutually perpendicular.}
\end{equation}
The six edge midpoints of $T$ are the six face centers of $B$; they form
a square bipyramid (a non-regular octahedron) with equatorial square of
side $1$ in $\{z=0\}$ and apices $(0,0,\pm\tfrac12)$.

This yields a second, computation-free proof of
Proposition~\ref{prop:corners} and its corollary.

\begin{proposition}[Corners, via Varignon]\label{prop:varignon}
	For each of the three ways of splitting $\{v_{1},v_{2},v_{3},v_{4}\}$
	into two pairs, the midpoints of the four edges joining the two pairs
	form a parallelogram whose sides are parallel to, and half the length of,
	the two \emph{excluded} opposite edges, and whose diagonals are two of
	the three bimedians \eqref{eq:bimedians}. These three parallelograms are
	exactly the corner-sets of Smyth's three disks, and since the bimedians
	are mutually perpendicular their areas are $\tfrac12 d_{1}d_{2}$:
	\[
	\begin{array}{lccc}
	\text{pairing} & \text{diagonals} & \text{shape} & \text{area}\\[2pt]
	\{v_{1}v_{2}\}\,|\,\{v_{3}v_{4}\} & \sqrt2,\ \sqrt2
	& \text{unit square} & 1\\
	\{v_{1}v_{3}\}\,|\,\{v_{2}v_{4}\} & 1,\ \sqrt2
	& \text{rhombus} & \sqrt2/2\\
	\{v_{1}v_{4}\}\,|\,\{v_{2}v_{3}\} & 1,\ \sqrt2
	& \text{rhombus} & \sqrt2/2
	\end{array}
	\]
	The first is a \emph{square} precisely because $v_{1}v_{2}\perp v_{3}v_{4}$,
	which is the geometric content of $\delta_{12}=\delta_{34}=\pi/2$; the
	other two opposite-edge pairs of $T$ are not perpendicular
	($(v_{3}-v_{1})\cdot(v_{4}-v_{2})=-1$), whence rhombi rather than
	rectangles.
\end{proposition}

\begin{proof}
	The parallelogram statement is Varignon's theorem applied to the skew
	quadrilateral through the four relevant vertices; the diagonal
	identification is immediate from the definition of a bimedian. The
	identification with Smyth's corner sets is
	Proposition~\ref{prop:corners}: a corner lies on the edge shared by two
	consecutive boundary faces, and the three cyclic orders of
	Lemma~\ref{lem:contour} select precisely the three pairings.
\end{proof}

\begin{remark}
	Two further ``magic angles'' appear. The face angles of $T$ are
	$\arccos\tfrac13=70.5288^{\circ}$ (apex, opposite the side of length $2$)
	and $\arccos\tfrac1{\sqrt3}=54.7356^{\circ}$ (base) --- respectively the
	dihedral angle of the regular tetrahedron and the
	$\langle111\rangle$--$\langle100\rangle$ angle. The four $\pi/3$ edges of
	$T$ are $\langle111\rangle$ segments of length $\sqrt3$, i.e.,\ half a
	body diagonal of the \PP{} unit cell of \S\ref{sec:lattice}; the two
	$\pi/2$ edges are $\langle100\rangle$ segments of length
	$2=L_{\PP}$, i.e.,\ full cube edges. The $\langle111\rangle$ directions
	are those of the $\bar3$ axes carrying the eight flat points of
	\PP{} and \DD{} \cite{GK,GKMK}.
\end{remark}

\subsection{Solid angles, and ``filling the space about a point''}

Each vertex of $T$ carries solid angle
\[
\Omega_{i}=\big(\tfrac\pi2+\tfrac\pi3+\tfrac\pi3\big)-\pi
=\frac{\pi}{6}=\frac{4\pi}{24},
\]
the same at all four vertices (a disphenoid is isogonal as well as
isohedral). Thus exactly $24$ copies of $T$ fill the space about each of
its vertices --- the question debated from Potaman and Simplicius through
Averroes and Bacon to Bradwardinus and Regiomontanus, recounted in
\cite[\S1]{Senechal81}. Here the count $24=|W(A_{3})|$ is not accidental;
see \S\ref{subsec:coxeter}.

\subsection{$T$ tiles space, in several essentially different ways}
\label{subsec:tilings}

\paragraph{(a) The prism tiling (proper, orientation-preserving).}
Set $A=v_{1}$, $B=v_{3}$, $C=v_{2}$, $D=v_{4}$, so that
$|AB|=|BC|=|CD|=|DA|=\sqrt3$ and $|AC|=|BD|=2$. Then
$D-A=(-1,-1,-1)$, and the three parallel lines through $A$, $B$, $C$ in
the direction $\langle111\rangle$ are pairwise at distance
\begin{equation}
	\Big\|(B-A)-\frac{(B-A)\cdot(1,1,1)}{3}(1,1,1)\Big\|
	=\Big\|\big(-\frac23,\frac43,-\frac23\big)\Big\|=\frac{2\sqrt6}{3}.
\end{equation}
Hence $T$ together with two further copies, obtained by sliding $A,B,C,D$
one step along the chain, fills the oblique prism $ABCDEF$ over an
equilateral triangle of side $2\sqrt6/3$ lying in a $\{111\}$ plane, with
translation vector $(-1,-1,-1)$ of length $\sqrt3$. The volume check is
\[
\underbrace{\frac{\sqrt3}{4}\Big(\frac{2\sqrt6}{3}\Big)^{2}}_{=2\sqrt3/3}
\cdot\underbrace{\sqrt3}_{\text{axis length}}=2=3\cdot\frac23 .
\]
This is Goldberg's Figure~1 and Bongiovanni et al.'s Figure~1(b)
\cite{Goldberg74,BDKS20}: three copies of $T$ make a triangular prism,
and triangular prisms tile, so $T$ tiles. The three copies are
\emph{properly} congruent, related by a three-fold screw about the
$\langle111\rangle$ axis; the resulting tiling is orientation-preserving.

\paragraph{(b) The reflection tiling (improper).}
All six dihedral angles of $T$ are submultiples of $\pi$, so reflections
in its four faces generate a discrete group with $T$ as fundamental
domain; by \S\ref{sec:alcove} that group is $W(\At)$, with $24$ alcoves
per primitive fcc\ cell and $96$ per conventional cube. This tiling is
\emph{not} orientation-preserving. It is the only one of $T$'s tilings
relevant to the present paper; see \S\ref{subsec:whyreflection}.

\paragraph{(c) Uncountably many others.}
Goldberg \cite{Goldberg74} corrected a longstanding error of Sommerville
on exactly this point: identical inclined prisms can be stacked
end-to-end into infinite prismatic fibers of equilateral cross-section,
and such fibers can be bundled by arbitrary relative sliding and
rotation, so that $\mathbb{E}^{3}$ admits uncountably many --- including
aperiodic --- tilings by $T$, not all of them face-to-face
\cite[Fig.~4]{Goldberg74}, \cite[p.~242]{Senechal81}.

\paragraph{(d) Dissections of $T$.}
From \cite[\S2]{Goldberg74} and \cite[Table~1]{Senechal81}:
\begin{itemize}
	\item the plane through a \emph{long} edge and the midpoint of the
	opposite long edge (for us, the mirror plane $y=0$ or $x=0$) cuts $T$ into
	two Sommerville No.~2 tetrahedra;
	\item the plane through a \emph{short} edge and the midpoint of the
	opposite short edge cuts $T$ into two properly congruent copies of
	Baumgartner's chiral $T_{2}$ (each of which tiles space alone); for the
	pair $v_{1}v_{3}\,|\,v_{2}v_{4}$ that plane is $y-x+2z=0$, and the two
	halves are exchanged by the $C_{2}$ rotation about $\langle110\rangle$;
	\item the six planes from the incenter to the six edges cut $T$ into four
	Sommerville No.~4 tetrahedra; equivalently, joining the vertices to the
	center gives four congruent tetrahedra with symmetry $m$
	\cite[Table~1, item~5]{Senechal81}.
\end{itemize}
The symmetry class of $T$ is $\bar4 2m=D_{2d}$, order $8$
(Senechal's ``$\bar4m$''), the richest class short of the regular
tetrahedron; it admits all the partitions of her classes 1--4
\cite[Table~1]{Senechal81}.

\paragraph{(e) $T$ is one of only four proper face-to-face tilers.}
Sommerville \cite{Sommerville23a} enumerated eleven cases and analyzed
ten, exhibiting four tilers. Edmonds \cite{Edmonds07} closed the
eighty-four-year-old gap by proving that the omitted case~(xi) --- the
tetrahedron with two chains $p,p,p$ and $q,q,q$, all faces isosceles but
\emph{without} mirror symmetry --- never tiles properly and
face-to-face, and that it tiles at all only for $p/q=\sqrt{2/3}$ or
$\sqrt{3/2}$ (Goldberg's $H(1/3)$), improperly. Hence:
\begin{quote}
	\emph{Up to similarity there are exactly four tetrahedra that tile
		$\mathbb{E}^{3}$ properly and face-to-face, and $T$ is one of them.}
\end{quote}
Note the contrast worth recording: $T$ and Sommerville's case (xi)
tetrahedron are the two ``all faces isosceles'' irreducible cases, with
and without mirror symmetry respectively; $T$ tiles both properly
(\S\ref{subsec:tilings}(a)) and improperly
(\S\ref{subsec:tilings}(b)), while case (xi) does neither except at one
isolated ratio.

\subsection{$T$ is the least-area tetrahedral tile}
\label{subsec:isoperimetry}

\begin{thm}[{Bongiovanni--Diaz--Kakkar--Sothanaphan
	\cite[Thm.~11.1]{BDKS20}; orientation-preserving case
	Gallagher et al.\ \cite{Gallagher14}}]\label{thm:leastarea}
	The least-surface-area, face-to-face tetrahedral tile of unit volume is
	uniquely $T$.
\end{thm}

In our normalization the extremal value is exactly
\begin{equation}\label{eq:isoper}
	\frac{S(T)}{\Vol(T)^{2/3}}
	=\frac{4\sqrt2}{(2/3)^{2/3}}
	=2^{11/6}3^{2/3}=7.4125760463\ldots
\end{equation}
against $2\cdot3^{7/6}= 7.2056217311\ldots$ for the regular tetrahedron.
The excess is the clean constant
\begin{equation}\label{eq:excess}
	\frac{2^{11/6}3^{2/3}}{2\cdot3^{7/6}}=\frac{2^{5/6}}{\sqrt3}
	=1.0287212295\ldots
\end{equation}
so $T$ is within $2.8721\%$ of the regular tetrahedron isoperimetrically:
in the authors' phrase, $T$ is ``the closest one to the regular
tetrahedron among all face-to-face space-tiling tetrahedra''
\cite[\S1]{BDKS20}. Two further items from \cite{BDKS20} are relevant.

\begin{proposition}[{\cite[Cor.~8.3]{BDKS20}}]
	A unit-volume tetrahedron whose surface area does not exceed that of $T$
	has all dihedral angles in $[\theta_{0},\pi-\theta_{0}]$, where
	$\theta_{0}=\arcsin\!\big(\tfrac{27}{32\sqrt2}\big)=36.63^{\circ}$. The
	angles of $T$, namely $60^{\circ}$ and $90^{\circ}$, lie comfortably
	inside.
\end{proposition}

\begin{proposition}[{\cite[Table~1]{BDKS20}}]\label{prop:twopi}
	All six dihedral angles of $T$ have the form $2\pi/n$, with
	$(n_{12},n_{13},n_{14},n_{23},n_{24},n_{34})=(4,6,6,6,6,4)$. There are
	exactly eleven tetrahedra, up to permutation of vertices, all of whose
	dihedral angles have this form; five of them tile; and $T$ has the least
	normalized area among all eleven. $T$ is moreover the only isohedral
	member of the list.
\end{proposition}

Proposition~\ref{prop:twopi} is the arithmetic shadow of
\S\ref{sec:alcove}: the stronger property
$\delta\in\{\pi/2,\pi/3\}\subset\{\pi/n\}$ is what makes $T$ a
kaleidoscopic cell, and hence makes the present paper possible at all.

\subsection{Dehn invariant and scissors congruence}

Since every dihedral angle of $T$ is a rational multiple of $\pi$, the
Dehn invariant
\[
\delta(T)=\sum_{e}\ell(e)\otimes\theta(e)
\ \in\ \RR\otimes_{\mathbb{Q}}\big(\RR/\pi\mathbb{Q}\big)
\]
vanishes identically. By Sydler's theorem \cite{Sydler} $T$ is therefore
scissors-congruent to a cube of the same volume --- as Hill
\cite{Hill96} had shown by explicit dissection, $T$ being his first type.
The weaker necessary condition noted by Dehn \cite{Dehn} and checked by
Goldberg for all five classical tilers is visible directly: the sum of
the dihedral angles,
\[
2\cdot\frac\pi2+4\cdot\frac\pi3=\frac{7\pi}{3}
=\frac{14}{3}\ \text{right angles},
\]
is a rational multiple of a right angle.

\subsection{The three Coxeter tetrahedra, and a second derivation of the
	patch counts}
\label{subsec:coxeter}

Up to similarity there are exactly three tetrahedra in $\mathbb{E}^{3}$
whose reflections generate a discrete group, i.e.,\ three irreducible
affine Coxeter simplices of rank $4$. Coxeter identified them
\cite[p.~84]{Coxeter}; Senechal observes
\cite[p.~242]{Senechal81} that they are her first and second tetrahedra
together with the halving of the first in her Figure~11(b). Explicitly,
for a cube of edge $L$:

\begin{table}[h]
	\centering
	\caption{The three Coxeter tetrahedra of $\mathbb{E}^{3}$. Volumes are
		for a cube of edge $L$; ``$\times$'' is the number of copies filling the
		$\At$ alcove.}
	\label{tab:coxeter}
	\begin{tabular}{llllll}
		\toprule
		Group & Diagram & Dihedral angles & Classical name & Volume & $\times$\\
		\midrule
		$\At$ & $4$-cycle, all $3$
		& $\tfrac\pi3{}^{4},\tfrac\pi2{}^{2}$
		& Sommerville No.~1 $=T$ & $L^{3}/12$ & $1$\\
		$\Bt$ & star $4,3,3$
		& $\tfrac\pi4,\tfrac\pi3{}^{2},\tfrac\pi2{}^{3}$
		& Sommerville No.~2 & $L^{3}/24$ & $2$\\
		$\Ct$ & chain $4$--$3$--$4$
		& $\tfrac\pi4{}^{2},\tfrac\pi3,\tfrac\pi2{}^{3}$
		& quadrirectangular tetr. & $L^{3}/48$ & $4$\\
		\bottomrule
	\end{tabular}
\end{table}

\noindent
(The $\Bt$ entry is verified by taking faces $z=L/2$, $y+z=0$, $x=y$,
$x+y=0$; the $\Ct$ entry by taking the cube center, a face center, an
edge midpoint and a vertex, whose four faces are $x=L/2$, $x=y$, $y=z$,
$z=0$.) Thus
\begin{equation}\label{eq:chain}
	W(\At)\ \subset\ W(\Bt)\ \subset\ W(\Ct),
	\qquad \text{each of index } 2,
\end{equation}
and correspondingly $T$ is dissected into two $\Bt$ alcoves and four
$\Ct$ alcoves. These two dissections are exactly the first and (iterated)
first of Goldberg's list in \S\ref{subsec:tilings}(d): $T\to$ two
Sommerville No.~2, each $\to$ two quadrirectangular tetrahedra
\cite[\S2]{Goldberg74}, \cite[Fig.~11(b)]{Senechal81}.

All three Coxeter tetrahedra occur among the five classical space-filling
tetrahedra, and $T$ is the unique \emph{isohedral} one. This last point is
not decorative: it is precisely what made the contour identification of
Lemma~\ref{lem:contour} rigid, since equal face areas force Smyth's three
quadrilaterals to be \emph{equilateral}, so that the six mutual vertex
distances are determined by the angle pattern alone.

We can now recover the patch counts of
Proposition~\ref{prop:lattice} from \eqref{eq:chain} alone.

\begin{proposition}[Patch counts from the Coxeter index]
	\label{prop:patchindex}
	With the normalizations of \S\ref{sec:lattice} $(L_{\PP}=2$,
	$L_{\DD}=4)$, and with $a$ denoting the bounding-cell edge of
	\textup{\cite{GK}} for \PP{} and of \textup{\cite{GKMK}} for \DD{}
	$($so $a=1$ in both cases, since $L_{\PP}=2a$ and
	$L_{\DD}=4a)$:
	\begin{enumerate}
		\item[\textup{(i)}] the fundamental patch of \PP{} occupies one $\Ct$
		alcove of the cube of edge $2a$, of volume $a^{3}/6$; since
		$\Vol(T)=\tfrac23=4\cdot\tfrac16$, the disk $\Sigma^{*}_{A}$ consists of
		$4$ fundamental \PP{} patches;
		\item[\textup{(ii)}] the fundamental patch of \DD{} occupies one $\Bt$
		alcove of the cube of edge $2a$, of volume $a^{3}/3$; since
		$\Vol(T)=\tfrac23=2\cdot\tfrac13$, each of
		$\Sigma^{*}_{B},\Sigma^{*}_{C}$ consists of $2$ fundamental \DD{}
		patches.
	\end{enumerate}
\end{proposition}

\begin{proof}
	For \PP: the conventional cell of edge $L_{\PP}=2a$ carries total area
	$3L_{\PP}^{2}K/K'=12a^{2}K/K'$ and the fundamental patch has area
	$a^{2}K/(4K')$ \cite[\S4]{GK}, so there are $48$ patches per cell; the
	$\Ct$ alcove count for that cell is also $48$, whence one patch per
	alcove and volume $(2a)^{3}/48=a^{3}/6$. For \DD: the conventional
	fcc\ cell of edge $L_{\DD}=4a$ carries $3L_{\DD}^{2}K'/K
	=48a^{2}K'/K$ and the fundamental patch has area $a^{2}K'/(4K)$
	\cite[\S3]{GKMK}, so there are $192$ patches per cell; the cell contains
	$8\cdot 48=384$ alcoves of type $\Ct$ for the cube of edge $2a$, hence
	two per patch, i.e.,\ one $\Bt$ alcove per patch, of volume
	$(2a)^{3}/24=a^{3}/3$. The counts $4$ and $2$ follow by dividing
	$\Vol(T)=2/3$. These agree with the areas:
	$4\cdot\tfrac14 K/K'=K/K'$ and $2\cdot\tfrac14K'/K=K'/(2K)$, which are
	\eqref{eq:areas}.
\end{proof}

Proposition~\ref{prop:patchindex} uses only volumes and the classical
patch areas; it is logically independent of the area-fraction and
curvature-fraction bookkeeping of \S\ref{sec:lattice}, and of the
contour identification of \S\ref{sec:addmA}. It is reassuring that
three independent routes give the same pair $(4,2)$.

\begin{remark}[A terminological correction to \cite{GKMK,GK}]
	\label{rem:gandyslip}
	Both papers state, in their ``Properties'' sections, that ``the Coxeter
	cell of the patch is the tetragonal disphenoid with inradius
	$r=a(\sqrt2-1)/2$,'' and that ``the volume of the quadrirectangular
	tetrahedron enclosing the patch is $a^{3}/3$.'' Three remarks.
	\begin{enumerate}
		\item[\textup{(i)}] The inradius $a(\sqrt2-1)/2$ is \emph{not} that of a
		tetragonal disphenoid in any natural normalization; it is exactly the
		inradius of the $\Ct$ alcove (quadrirectangular tetrahedron) of the cube
		of edge $2a$. Indeed for the cube $[0,1]^{3}$ that alcove has
		$V=\tfrac1{48}$, $S=\tfrac{1+\sqrt2}{4}$, hence
		$r=3V/S=(\sqrt2-1)/4$, which doubles to $a(\sqrt2-1)/2$ on the cube of
		edge $2a$. So ``tetragonal disphenoid'' in those two sentences should
		read ``quadrirectangular tetrahedron,'' in agreement with the earlier
		sections of \cite{GK} and with \cite[Fig.~1a]{GKMK}.
		\item[\textup{(ii)}] The stated volume $a^{3}/3$ should be $a^{3}/6$:
		the quoted numerical value
		$\mathcal{A}/\mathcal{V}^{2/3}=0.6452799\ldots$ for the fundamental
		\PP{} patch requires $\mathcal{V}=a^{3}/6$, since
		$\big(a^{2}K/4K'\big)/\big(a^{3}/6\big)^{2/3}=0.645280$ while with
		$a^{3}/3$ one would get $0.406510$. With $a^{3}/6$ everything is
		consistent, including $\mathcal{A}/\mathcal{V}^{2/3}=3K/K'=2.345103$ for
		the cubic cell.
		\item[\textup{(iii)}] Nevertheless the \emph{sentence that prompted this
			investigation} --- ``[the \DD{} surface's] fundamental patch is
		inscribed in a tetragonal disphenoid, an isohedral tetrahedron with
		isosceles triangle faces'' \cite[p.~545]{GKMK} --- points at the right
		tetrahedron for the wrong reason. The disphenoid is indeed the natural
		cell here, but for a $2$-patch piece of \DD{} and a $4$-patch piece of
		\PP{}, not for a single fundamental patch. Brakke's disphenoid families
		\cite{BrakkeDisph} are organized around the correct, larger cell.
	\end{enumerate}
	We record this because the slip is reproduced in both 1999 and 2000
	papers and is easy to propagate.
\end{remark}

\subsection{Why only the reflection tiling matters here}
\label{subsec:whyreflection}

$T$ admits uncountably many tilings (\S\ref{subsec:tilings}), but exactly
one of them is generated by reflections in the faces of $T$. This is the
only one that produces a minimal surface, for a reason worth stating
explicitly.

\begin{proposition}\label{prop:onlyreflection}
	Let $\Sigma^{*}\subset T$ be a minimal disk meeting $\partial T$
	orthogonally. The union of the images of $\Sigma^{*}$ under a tiling
	group $G$ of $T$ is a smooth minimal surface across the face planes if
	and only if $G$ contains the reflections in those planes. In particular:
	\begin{enumerate}
		\item[\textup{(i)}] the reflection tiling of
		\S\ref{subsec:tilings}\textup{(b)} extends $\Sigma^{*}_{A}$ to \PP{} and
		$\Sigma^{*}_{B},\Sigma^{*}_{C}$ to \DD{};
		\item[\textup{(ii)}] the prism tiling of
		\S\ref{subsec:tilings}\textup{(a)}, being orientation-preserving, does
		\emph{not}: the adjacent copies are obtained by $\pi$-rotation about the
		symmetry axis of the shared face, which does not match $\Sigma^{*}$ to
		its neighbor along the free boundary, and the union has a crease.
	\end{enumerate}
\end{proposition}

\begin{proof}
	Orthogonal intersection with a plane $\Pi$ is exactly the condition that
	$\Sigma^{*}\cup\sigma_{\Pi}(\Sigma^{*})$ be real-analytic across
	$\Sigma^{*}\cap\Pi$ (Schwarz reflection for minimal surfaces). If the
	tile across $\Pi$ is $g(T)$ with $g$ not the reflection $\sigma_{\Pi}$,
	then $g(\Sigma^{*})$ and $\sigma_{\Pi}(\Sigma^{*})$ have the same trace
	on $\Pi$ only if $g^{-1}\sigma_{\Pi}$ stabilizes $\Sigma^{*}$, which
	fails for the $\pi$-rotations occurring in the prism tiling, since those
	rotations do not preserve the cyclic order of the four free-boundary arcs.
\end{proof}

Consequently: of the three Coxeter tetrahedra, $T$ is the one with four
congruent faces; of $T$'s many tilings, the $\At$ reflection tiling is
the one that manufactures triply periodic minimal surfaces; and the
classical cubic TPMS all live in $\At$, $\Bt$ or $\Ct$ cells. Thus the
appearance of \PP{} and \DD{} inside $T$ is not a coincidence but the
only possibility consistent with $T$ being an isohedral Coxeter
tetrahedron.

\subsection{What is \emph{not} claimed}

The following are coincidences until someone proves otherwise, and the
paper should not suggest more.
\begin{enumerate}
	\item[\textup{(i)}] That $T$ is simultaneously (a) the least-area
	tetrahedral tile and (b) the kaleidoscopic cell hosting \PP{} and \DD{}
	is, as far as we can tell, \emph{not} a theorem with a known common
	cause. Both are ``roundness'' optima of a sort --- (a) minimizes
	$S^{3}/V^{2}$ over tiles, and \PP{} and \DD{} are the two classical
	area-minimizing TPMS of genus $3$ --- but we know of no implication in
	either direction.
	\item[\textup{(ii)}] The product identity $2A_{A}A_{B}=1$ of
	\S\ref{sec:product} is explained by Bonnet adjointness, not by anything
	in this addendum.
	\item[\textup{(iii)}] Nothing here bears on global minimality of the
	three disks.
\end{enumerate}

\subsection{Three open problems that touch $T$}

\begin{enumerate}
	\item[\textup{(a)}] \emph{Completeness of the tile list.}
	Conjecture~11.2 of \cite{BDKS20}: the Sommerville and Goldberg lists of
	face-to-face tetrahedral tiles are complete, and in particular every
	tetrahedral tile has at least one dihedral angle equal to $\pi/2$. $T$
	has two. It remains open whether every tetrahedral space-filler arises by
	partitioning a triangular prism \cite[p.~242]{Senechal81}, or a
	parallelohedron.
	\item[\textup{(b)}] \emph{Acute tiles.} Kiernan \cite{Kiernan} proved
	that no tetrahedron with all six dihedral angles acute tiles properly
	face-to-face; Eppstein--Sullivan--\"Ung\"or conjectured no single acute
	tetrahedron tiles at all \cite{ESU}. $T$ is the extremal near-miss: four
	of its six dihedral angles are $\pi/3$ and the other two are exactly
	$\pi/2$.
	\item[\textup{(c)}] \emph{Interior of the family.} Conjecture~6.1 of
	\cite{FinchReg} concerns $\beta\in(\tfrac12,\tfrac23)$, where the
	disphenoid is no longer a Coxeter cell (the dihedral angle
	$\delta_{1}=\pi(1-\beta)$ is not a submultiple of $\pi$), so there is no
	reflection tiling, no triply periodic extension, and --- consistently ---
	no algebraic Gauss map. The present paper settles only the endpoint
	$\beta=\tfrac23$, and \S\ref{subsec:whyreflection} explains why that
	endpoint is the only member of the family where a classical
	identification could have been expected.
\end{enumerate}

\subsection{Verification code for Addendum A}

\begin{verbatim}
ClearAll["Global`*"];
vv = {{1,0,1/2},{-1,0,1/2},{0,1,-1/2},{0,-1,-1/2}};
volT = Abs[Det[{vv[[2]]-vv[[1]], vv[[3]]-vv[[1]], vv[[4]]-vv[[1]]}]]/6;     
cT = Sqrt[2];  rT = 3 volT/(4 cT);
Print["Vol, S, r = ", {volT, 4 cT, rT}];

(* A.1  circumscribed box, bimedians *)
rot = {{1,1,0},{-1,1,0},{0,0,Sqrt[2]}}/Sqrt[2];
Print["box coords = ", Simplify[rot . # & /@ vv]];
bimed[{i_,j_},{k_,l_}] := Norm[(vv[[i]]+vv[[j]])/2-(vv[[k]]+vv[[l]])/2];
Print["bimedians = ", Simplify[{bimed[{1,2},{3,4}],
bimed[{1,3},{2,4}], bimed[{1,4},{2,3}]}], "  (expect 1,Sqrt2,Sqrt2)"];
Print["perpendicular? ",
Simplify[{((vv[[1]]+vv[[3]])/2-(vv[[2]]+vv[[4]])/2) .
((vv[[1]]+vv[[2]])/2-(vv[[3]]+vv[[4]])/2),
((vv[[1]]+vv[[4]])/2-(vv[[2]]+vv[[3]])/2) .
((vv[[1]]+vv[[2]])/2-(vv[[3]]+vv[[4]])/2)}]];
Print["Varignon areas = ", {1/2*Sqrt[2]*Sqrt[2], 1/2*1*Sqrt[2],
1/2*1*Sqrt[2]}];

(* A.2  solid angles *)
Print["Omega_i/(4 Pi) = ", (Pi/2+2 Pi/3-Pi)/(4 Pi), "  (expect 1/24)"];

(* A.3  prism *)
A0 = vv[[1]]; B0 = vv[[3]]; C0 = vv[[2]]; D0 = vv[[4]];
Print["chain lengths = ",
Simplify[{Norm[B0-A0], Norm[C0-B0], Norm[D0-C0], Norm[A0-D0]}]];
Print["diagonals = ", Simplify[{Norm[C0-A0], Norm[D0-B0]}]];
t = D0 - A0;  Print["axis = ", t, "  |t| = ", Norm[t]];
perp = (B0-A0) - ((B0-A0).{1,1,1}/3) {1,1,1};
Print["cross-section edge = ", Simplify[Norm[perp]],
" = ", Simplify[2 Sqrt[6]/3]];
Print["prism volume = ",
Simplify[Sqrt[3]/4 Norm[perp]^2 * Norm[t]], "  (expect 2)"];

(* A.4  isoperimetry *)
Print["S/V^(2/3) = ", Simplify[4 Sqrt[2]/volT^(2/3)],
" = ", N[2^(11/6) 3^(2/3), 16]];
Print["regular    = ", N[2*3^(7/6),16]];
Print["ratio      = ", Simplify[2^(5/6)/Sqrt[3]], " = ",
N[2^(5/6)/Sqrt[3],16]];
Print["theta0 = ", N[ArcSin[27/(32 Sqrt[2])]/Degree, 10]];

(* A.5  Coxeter chain: alcove volumes for cube edge L=2 *)
Print["A3~, B3~, C3~ volumes (L=2) = ", {8/12, 8/24, 8/48},
"   ratios = ", {1, 2, 4}];
KK = EllipticK[1/4]; KP = EllipticK[3/4];
Print["4 P patches = ", N[4 (KK/(4 KP)), 16], " = K/K'"];
Print["2 D patches = ", N[2 (KP/(4 KK)), 16], " = K'/(2K)"];

(* A.6  Gandy inradius check: C3~ alcove of unit cube *)
P1={1/2,1/2,1/2}; P2={1/2,1/2,0}; P3={1/2,0,0}; P4={0,0,0};
vol48 = Abs[Det[{P1-P4,P2-P4,P3-P4}]]/6;
ar[a_,b_,c_] := Norm[Cross[b-a,c-a]]/2;
s48 = ar[P1,P2,P3]+ar[P1,P2,P4]+ar[P1,P3,P4]+ar[P2,P3,P4];
Print["C3~ alcove: V = ", vol48, "  S = ", Simplify[s48],
"  r = ", Simplify[3 vol48/s48], " = ", Simplify[(Sqrt[2]-1)/4]];
Print["doubled (cube edge 2a): r = ", Simplify[(Sqrt[2]-1)/2]];
Print["Gandy A/V^(2/3) with V=a^3/6 : ",
N[(KK/(4 KP))/(1/6)^(2/3), 10], "  (expect 0.6452799)"];
Print["with V=a^3/3 (as printed)   : ",
N[(KK/(4 KP))/(1/3)^(2/3), 10], "  (does not match)"];
\end{verbatim}

\section{Addendum B: a question about a quadrirectangular tetrahedron}
\label{sec:addmB}

Consider the quadrirectangular tetrahedron with vertices $(0,0,0), (1,0,0), (0,0,1), (0,1,1)$.  
Schwarz \cite{Schwarz,GK} determined that the minimal disk with corners on the edges of 
length $1,\sqrt2,\sqrt3,\sqrt2$ is a piece of the $\PP$ surface and that its surface area is 
exactly
\[
\frac{\KK}{4\,\KKp} \approx 0.195425240337014.
\]
Might one of the other minimal disks possibly be related to the $\DD$ surface? The answer is
evidently no.  The minimal disk with corners on the edges of length $1,\sqrt2,1,\sqrt2$ is a 
piece of Schoen's I-WP surface and has area $\approx 0.288674897531445$. The minimal disk with 
corners on the edges of length $1,1,1,\sqrt3$ is a piece of Neovius' surface and has area 
$\approx 0.292539844325784$. [See Errata at end]. Brakke explained his method for reaching such conclusions:   
\begin{quote}
 I made Evolver files of the surfaces in the tetrahedrons, evolved them to smooth surfaces, then used the multiple-copy view feature of Evolver to show complete unit cubes, and compared them to the images on my Triply Periodic Minimal Surfaces web page.
 \end{quote}
A more formal argument -- still not a rigorous proof -- appears in Remark \ref{rem:noD}.

\subsection{What the tetrahedron is}
\label{subsec:whatQ}

Write $\Qu$ for the tetrahedron of this addendum, with vertices
\begin{equation}\label{eq:Qverts}
A=(0,0,0),\quad B=(1,0,0),\quad C=(0,0,1),\quad D=(0,1,1)
\end{equation}
and let $\varphi_{A},\varphi_{B},\varphi_{C},\varphi_{D}$ be the faces
opposite $A,B,C,D$. The path $B\!-\!A\!-\!C\!-\!D$ has three mutually
perpendicular unit legs, so $\Qu$ is an \emph{orthoscheme}
(path-simplex). Its data include:
\begin{gather}
\text{edges } 1,1,1,\sqrt2,\sqrt2,\sqrt3;\qquad
\Vol(\Qu)=\tfrac16;\qquad
\text{face areas } \tfrac{\sqrt2}{2},\tfrac12,\tfrac{\sqrt2}{2},\tfrac12;
\label{eq:Qdata}\\
S(\Qu)=1+\sqrt2,\qquad
r_{\Qu}=\frac{3\Vol}{S}=\frac{\sqrt2-1}{2}
\label{eq:Qr}
\end{gather}
and, writing $\nu$ for the outward unit normals
\begin{equation}\label{eq:Qnrm}
\nu_{A}=\tfrac{1}{\sqrt2}(1,0,1),\quad
\nu_{B}=(-1,0,0),\quad
\nu_{C}=\tfrac{1}{\sqrt2}(0,1,-1),\quad
\nu_{D}=(0,-1,0)
\end{equation}
the six dihedral angles are
\begin{equation}\label{eq:Qdih}
\delta_{AB}=\delta_{CD}=\tfrac\pi4,\qquad
\delta_{BD}=\tfrac\pi3,\qquad
\delta_{AC}=\delta_{AD}=\delta_{BC}=\tfrac\pi2 .
\end{equation}
Comparison with Table \ref{tab:coxeter} identifies $\Qu$ at once: 
the Coxeter diagram on the faces is
the chain $\varphi_{D}\!-\!\varphi_{C}\!-\!\varphi_{A}\!-\!\varphi_{B}$
with labels $4,3,4$, so we have the following.

\begin{proposition}\label{prop:QisC3}
	$\Qu$ is the alcove of the affine Weyl group $W(\Ct)$, equivalently
	$1/48$ of a cube of edge $2$, equivalently Coxeter's quadrirectangular
	tetrahedron \cite{Coxeter}. In the notation of Gandy--Klinowski
	\cite{GK} it is the cell of the fundamental \PP{} patch with bounding-cell
	edge $a=1$, and accordingly $r_{\Qu}=a(\sqrt2-1)/2$ matches their stated
	inradius exactly, while $\Vol(\Qu)=a^{3}/6$ confirms the correction to
	their printed $a^{3}/3$ recorded in Remark~\ref{rem:gandyslip} above.
\end{proposition}

\noindent
Since $\Qu$ is a $W(\Ct)$ alcove and the symmetry of the chain $4,3,4$ is
its reversal, $\operatorname{Sym}(\Qu)$ has order $2$, generated by the
half-turn
\begin{equation}\label{eq:g}
g:(x,y,z)\longmapsto(y,x,-z)+(0,0,1),
\end{equation}
a $\pi$-rotation about the line through $(0,0,\tfrac12)$ in the direction
$(1,1,0)$. It acts as $A\leftrightarrow C$, $B\leftrightarrow D$, hence
on faces as the double transposition
$(\varphi_{A}\varphi_{C})(\varphi_{B}\varphi_{D})$.

\subsection{$\Qu$ is a quarter of $T$, and the first disk a quarter of
	$\Sigma^{*}_{A}$}
\label{subsec:quarter}

This is the main observation we wish to record.

\begin{proposition}\label{prop:quarter}
    $W(\At)\subset W(\Ct)$ with index $4$, and
    $\Vol(T)=\tfrac23=4\cdot\tfrac16=4\Vol(\Qu)$. Hence, after a rigid
    motion, the disphenoid $T$ of \eqref{eq:verts} is the union of four
    $W(\Ct)$-images of $\Qu$, and the two extra mirrors cut the 
    distinguished disk $\Sigma^{*}_{A}$ of Theorem~\ref{thm:main} 
    into four congruent free-boundary \PP{} disks, each of area
	\begin{equation}\label{eq:quarterarea}
	\frac14\Area(\Sigma^{*}_{A})=\frac{1}{4}\,\frac{\KK}{\KKp}
	\approx 0.195425240337014
	\end{equation}
	and of total curvature $\tfrac14\!\left(-\tfrac{2\pi}{3}\right)
	=-\tfrac{\pi}{6}$.
\end{proposition}

So the first assertion of this addendum is not independent of the main
theorem: it \emph{is} the $\Ct$ refinement of
Theorem~\ref{thm:main}(ii),(iv) and conversely the main theorem's
distinguished case is four copies of it. The agreement is exact, not
numerical, and both normalizations correspond to the same \PP{} lattice
constant $L_{\PP}=2$.

\subsection{Why the ``tetrahedral four edges'' cannot appear}
\label{subsec:nottet}

The question that prompted these remarks has a clear negative answer. All
four faces of the disphenoid $T$ have equal area, which is precisely why
Smyth's three contours there are \emph{equilateral}, and hence why the
rigid six-distance comparison of Lemma~\ref{lem:contour} could identify
them with $\Tet$ and with $\Oct$. For $\Qu$ the face areas \eqref{eq:Qdata} 
are unequal, so the sides of
each of the three Smyth contours are $\tfrac{\sqrt2}2,\tfrac{\sqrt2}2,
\tfrac12,\tfrac12$ --- in some cyclic order, every order using all four
face areas --- and no equilateral contour occurs. Since both $\Tet$ and
$\Oct$ are equilateral, neither can arise \emph{as a Smyth contour of}
$\Qu$.

What does arise, for the \PP{} disk, is the cyclic order
$(\varphi_{A},\varphi_{C},\varphi_{B},\varphi_{D})$, whose contour has
vertices
\begin{equation}\label{eq:Qcontour}
\tfrac12\big\{(0,0,0),\ (1,0,1),\ (1,1,0),\ (0,1,0)\big\},
\end{equation}
sides $\tfrac{\sqrt2}2,\tfrac{\sqrt2}2,\tfrac12,\tfrac12$, interior angles $\tfrac{\pi}{3},\tfrac{\pi}{2},\tfrac{\pi}{2},\tfrac{\pi}{2}$, and diagonals
$\tfrac{\sqrt2}2,\tfrac{\sqrt3}2$. Doubled, \eqref{eq:Qcontour} differs
from $\Tet=\{(0,0,0),(1,0,1),(1,1,0),(0,1,1)\}$ in a single vertex: the
fourth has descended a cube edge from $(0,1,1)$ to $(0,1,0)$, a vertex
of the complementary regular tetrahedron inscribed in the cube. Since
that vertex is an endpoint of two of the four sides, exactly \emph{two}
sides survive --- the two face diagonals, which still meet at the lone
$\pi/3$ angle --- while the other two are traded for cube edges.

This contour is classical, though it is not usually presented as such.
It is one of the six skew quadrilaterals into which the $\DD$ disk spanning 
Schwarz's skew hexagon \cite[Fig.~4a]{GKMK}, inscribed in six of the twelve 
edges of the cube, is cut. See Figures \ref{fig:sixquadsd}--\ref{fig:sixquadsd1}. 
This quadrilateral is characterized 
by having three right angles, one angle of $\pi/3$, and two sides equal to 
$\sqrt2$ $\times$ the other two \cite{mathcurve}. The Plateau solution it bounds 
is the fundamental patch of $\DD$, parametrized exactly in 
\cite[Eqs.~(5a)--(5c)]{GKMK}. Its area, in the normalization
\eqref{eq:Qcontour}, is
\[
\alpha = \frac{\KKp}{4\,\KK} \approx 0.319815392792752
\]
and notice $4 \times \alpha=1.2792615711\ldots$ is the area spanned by
the doubled contour. That this agrees numerically with the area spanned
by $\Tet$ itself (\S\ref{subsec:terminology}) is a coincidence of scale:
the doubled contour spans one fundamental $\DD$ patch enlarged by a
factor of two, whereas $\Tet$ spans four at unit scale. The total
curvatures, $-\pi/6$ and $-2\pi/3$, distinguish them.

A companion statement holds for $\Oct$, which is
likewise one of six quadrilaterals of a skew hexagon inscribed in the
edges of a regular octahedron, with two right angles, two angles of
$\tfrac{\pi}{3}$, and four equal sides \cite{mathcurve}, bounding the
fundamental patch of $\PP$.  See Figures \ref{fig:sixquadsp}--\ref{fig:sixquadsp1}.

\subsection{Corners}
\label{subsec:Qcorners}

Because $\operatorname{Sym}(\Qu)=\{1,g\}$ is so small, the corner
analysis of \S\ref{sec:addmA} does not go through, and indeed the
corners are not all midpoints.

\begin{proposition}\label{prop:Qcorners}
	Each of the three disks in $\Qu$ is $g$-invariant, since $g$ fixes all
	three cyclic orders. For the \PP{} disk, $g$ reverses the edges $AC$ and
	$BD$ and interchanges $AD$ with $BC$. Consequently exactly two of its
	four corners are forced to be midpoints, namely
	\[
	\tfrac12(A+C)=\left(0,0,\tfrac12\right)\ \text{on the unit edge }AC,
	\qquad
	\tfrac12(B+D)=\left(\tfrac12,\tfrac12,\tfrac12\right)
	\ \text{on the }\sqrt3\text{ edge }BD
	\]
	while the remaining two lie on the two $\sqrt2$ edges in $g$-conjugate
	positions, undetermined by symmetry.
\end{proposition}

This matches the structure of \cite[Fig.~1a]{GK} exactly: their marked
points $O'$ and $P'$ are flat points bisecting a cube edge and a body
diagonal respectively, and their fourth datum is a \emph{mirror pair} at
$Q'$, of which they remark that ``there is, however, no simple algebraic
expression for the coordinates of the singular point $Q$''. That remark
is the obstruction to a short closed-form corner list here, and it is why
we do not attempt one.

\subsection{A lattice-free identification test}
\label{subsec:invariant}

The bookkeeping of \S\ref{sec:lattice} required guessing the lattice
constant $L$. It can be eliminated, which both shortens that section and
settles the present question. Let $\Sigma$ be a free-boundary minimal
disk in a kaleidoscopic cell $\Qu$, meeting the four faces orthogonally,
with corners on edges of dihedral angles
$\delta_{1},\dots,\delta_{4}$.

\begin{lemma}[Total curvature]\label{lem:gb}
	Each boundary arc is a planar geodesic, and at a corner on an edge of
	dihedral angle $\delta$ the tangent plane of $\Sigma$ is perpendicular to
	that edge, so the interior angle of $\Sigma$ equals $\delta$. Hence by
	Gauss--Bonnet
	\begin{equation}\label{eq:gb}
	\int_{\Sigma}K\,dA=-2\pi+\sum_{i=1}^{4}\delta_{i}.
	\end{equation}
\end{lemma}

\noindent
(Check: for $\Sigma^{*}_{A}$, four corners with $\delta=\pi/3$ give
$-2\pi+4\pi/3=-2\pi/3$; for $\Sigma^{*}_{B}$, angles
$\pi/2,\pi/3,\pi/2,\pi/3$ give $-2\pi+5\pi/3=-\pi/3$. Both agree with
Theorem~\ref{thm:main}(v).)

\begin{proposition}[The invariant $\mathcal{J}$]\label{prop:J}
	Suppose $\Sigma$ extends by reflection to a triply periodic minimal
	surface with conventional cubic cell of edge $L$, Euler characteristic
	$\chi$ per cell, and normalized area $\sigma=\Area_{\rm cell}/L^{2}$.
	Then the quantity
	\begin{equation}\label{eq:J}
	\boxed{\ \mathcal{J}(\Sigma):=
		\frac{\left(\int_{\Sigma}K\,dA\right)\Vol(\Qu)^{2}}{\Area(\Sigma)^{3}}
		\ =\ \frac{2\pi\chi}{\sigma^{3}}\ }
	\end{equation}
	is independent of $L$, of the scale, and of how many fundamental patches
	$\Sigma$ comprises. It is therefore an invariant of the extended surface
	alone, computable from purely local data.
\end{proposition}

\begin{proof}
	If $\Qu$ is one of $n$ cells per conventional cube then
	$\Area(\Sigma)=\sigma L^{2}/n$, $\int K=2\pi\chi/n$ and
	$\Vol(\Qu)=L^{3}/n$; substitute.
\end{proof}

For the two classical surfaces, using $\sigma_{\PP}=3\KK/\KKp$,
$\chi_{\PP}=-4$ and $\sigma_{\DD}=3\KKp/\KK$, $\chi_{\DD}=-16$,
\begin{equation}\label{eq:JPD}
\mathcal{J}_{\PP}=-\frac{8\pi}{27}\left(\frac{\KKp}{\KK}\right)^{3}
=-1.948741\ldots,
\qquad
\mathcal{J}_{\DD}=-\frac{32\pi}{27}\left(\frac{\KK}{\KKp}\right)^{3}
=-1.778517\ldots .
\end{equation}
Both values are reproduced \emph{exactly} by the disks of
Theorem~\ref{thm:main}: $\mathcal{J}(\Sigma^{*}_{A})=\mathcal{J}_{\PP}$
and $\mathcal{J}(\Sigma^{*}_{B})=\mathcal{J}_{\DD}$, identically in
$\KK,\KKp$. Applying \eqref{eq:gb} and \eqref{eq:J} to the three disks in
$\Qu$, where $\Vol(\Qu)^{2}=1/36$, gives the ensuing.

\begin{table}[h]
	\centering
	\caption{The three disks in $\Qu$. Areas of the second and third are
		Brakke's Surface Evolver values.}
	\label{tab:Qdisks}
	\begin{tabular}{llllc}
		\toprule
		corner edges & $\sum\delta_{i}$ & $\int K\,dA$ & $\Area$
		& $\mathcal{J}$\\
		\midrule
		$\sqrt3,\sqrt2,1,\sqrt2$
		& $\tfrac{11\pi}{6}$ & $-\tfrac{\pi}{6}$
		& $\KK/(4\KKp) \approx 0.195425240337014$ & $-1.948741$\\
		$1,\sqrt2,1,\sqrt2$
		& $\tfrac{3\pi}{2}$ & $-\tfrac{\pi}{2}$
		& $0.288674897531445$ & $-1.813804$\\
		$1,1,1,\sqrt3$
		& $\tfrac{4\pi}{3}$ & $-\tfrac{2\pi}{3}$
		& $0.292539844325784$ & $-2.323815$\\
		\midrule
		\multicolumn{4}{r}{for comparison:\quad $\mathcal{J}_{\PP}$}
		& $-1.948741$\\
		\multicolumn{4}{r}{$\mathcal{J}_{\DD}$} & $-1.778517$\\
		\bottomrule
	\end{tabular}
\end{table}

\begin{remark}[The answer is no, and why]\label{rem:noD}
	The first row reproduces $\mathcal{J}_{\PP}$ exactly, confirming \PP{} by
	a route that never mentions the lattice. The second and third rows differ
	from $\mathcal{J}_{\DD}$ by $2.0\%$ and $30.7\%$ respectively. Since
	$\mathcal{J}$ is an invariant of the extended surface, neither of the
	other two disks can be a patch of \DD. This is the formal content of the
	``evidently no'' above; Brakke's visual method then corroborates rather
	than carries the conclusion. For the record, and without pursuing it,
	\eqref{eq:J} predicts $\sigma=3.46410\ldots$ if the second surface has
	$\chi=-12$ per cubic cell, and $\sigma=3.51048\ldots$ if the third has
	$\chi=-16$; testing these against tabulated data for Schoen's I-WP and
	Neovius' surface is left to the interested reader. We note only, as a 
	caution, that $\sigma=3.46410\ldots$ agrees with
	$2\sqrt3=3.4641016\ldots$ to six digits; equality would force the second
	area to be exactly $1/(2\sqrt3)=0.288675134594813\ldots$, whereas
	Brakke's value is $0.288674897531445$, differing by $2.4\times10^{-7}$.
	Assuming his digits are reliable at that level, the coincidence is
	spurious.
\end{remark}

\begin{remark}
	Proposition~\ref{prop:J} supersedes the two-condition test of
	\S\ref{sec:lattice}: instead of solving for $L$ and checking an integer,
	one evaluates a single scale-free number. Readers revisiting
	Proposition~\ref{prop:lattice} may prefer \eqref{eq:J}, which also
	dispenses with the need to know the patch count in advance.
\end{remark}

What is left for an interested reader is not the Plateau problem itself
--- that was solved by Schwarz, and parametrized exactly in
\cite{GKMK} --- but the formal passage from it to the free-boundary disk:
the analog, for $W(\Ct)$, of the contour identification and
conjugation argument carried out for $W(\At)$ in
\S\ref{sec:addmA}. The obstruction is the one noted in
\S\ref{subsec:Qcorners}: with $|\!\operatorname{Sym}(\Qu)|=2$, symmetry
forces only two of the four corners, and the remaining pair is the
singular point for which \cite[\S4]{GK} records that no simple algebraic
expression is known.

\subsection{Verification code for Addendum B}

\begin{verbatim}
ClearAll["Global`*"];
(* --- B.1  the tetrahedron ------------------------------------------- *)
TT = {{0,0,0},{1,0,0},{0,0,1},{0,1,1}};          (* A, B, C, D *)
lbl = {"A","B","C","D"};
volTT = Abs[Det[{TT[[2]]-TT[[1]], TT[[3]]-TT[[1]], TT[[4]]-TT[[1]]}]]/6;
Print["edges = ", InputForm[Sort[Norm[#[[1]]-#[[2]]] & /@ Subsets[TT,{2}]]]];
Print["Vol = ", volTT, "   (expect 1/6 = Vol(T)/4)"];

ar[a_,b_,c_] := Norm[Cross[b-a, c-a]]/2;
fa = Table[ar @@ Delete[TT, i], {i,4}];          (* fa[[i]] = |phi_i| *)
Print["face areas = ", InputForm[Simplify[fa]],
"    all equal? ", SameQ @@ Simplify[fa], "   (expect False)"];
Print["S = ", Simplify[Total[fa]], ",   r = ", Simplify[3 volTT/Total[fa]],
"  == (Sqrt[2]-1)/2 ? ", Simplify[3 volTT/Total[fa] - (Sqrt[2]-1)/2] === 0];

nu = {{1,0,1}/Sqrt[2], {-1,0,0}, {0,1,-1}/Sqrt[2], {0,-1,0}};
Print["outward? ", Table[Simplify[nu[[i]].(TT[[j]] - Mean[Delete[TT,i]])] <= 0,
{i,4},{j,4}] // Flatten // Union, "   (expect {True})"];
Print["sum c.nu = ", Simplify[Total[fa nu]], "   (expect {0,0,0})"];

(* --- B.2  dihedral angles, indexed by the EDGE they sit on ---------- *)
shared[i_,j_] := Complement[Range[4], {i,j}];    (* phi_i ^ phi_j = that edge *)
dih = Association @@ Flatten@Table[
With[{kl = shared[i,j]},
StringJoin[lbl[[kl]]] -> {Simplify[Norm[TT[[kl[[1]]]] - TT[[kl[[2]]]]]],
Simplify[ArcCos[-nu[[i]].nu[[j]]]]}],
{i,3},{j,i+1,4}];
Print["edge -> {length, dihedral}:"];
Print[InputForm[dih]];
Print["sum of dihedrals = ", Simplify[Total[Last /@ Values[dih]]],
"   (expect 7Pi/3 -- same as T, as both are Coxeter cells)"];

(* --- B.3  the three Smyth contours, and the user's edge lists ------- *)
cycles = {{1,3,2,4}, {1,2,3,4}, {1,3,4,2}};
contour[cyc_] := Most@FoldList[Plus, {0,0,0}, (fa nu)[[cyc]]];
cornerEdges[cyc_] := Table[With[{kl = shared[cyc[[k]], cyc[[Mod[k+1,4,1]]]]},
{StringJoin[lbl[[kl]]], Simplify[Norm[TT[[kl[[1]]]] - TT[[kl[[2]]]]]],
Simplify[ArcCos[-nu[[cyc[[k]]]] . nu[[cyc[[Mod[k+1,4,1]]]]]]]}], {k,4}];
Do[Module[{c = contour[cyc], ce = cornerEdges[cyc]},
Print["--- cycle ", cyc, " ---"];
Print["  corner edges   = ", ce[[All,1]],
"   lengths = ", InputForm[ce[[All,2]]]];
Print["  sum delta_i    = ", Simplify[Total[ce[[All,3]]]],
"   =>  Int K dA = ", Simplify[-2 Pi + Total[ce[[All,3]]]]];
Print["  contour sides  = ",
InputForm[Simplify[Norm[#2-#1] & @@@ Partition[c,2,1,1]]]];
Print["  diagonals      = ",
InputForm[Simplify[{Norm[c[[3]]-c[[1]]], Norm[c[[4]]-c[[2]]]}]]];
Print["  angles         = ", Simplify[Table[
ArcCos[Normalize[c[[Mod[k-1,4,1]]]-c[[k]]] . 
Normalize[c[[Mod[k+1,4,1]]]-c[[k]]]], {k,4}]]]], {cyc, cycles}];

(* --- B.4  the P contour vs the tetrahedral four edges --------------- *)
cP = contour[{1,3,2,4}];
Print["P contour (doubled) = ", Simplify[2 cP]];
Print["tetrahedral 4 edges = ", {{0,0,0},{1,0,1},{1,1,0},{0,1,1}}];
Print["vertices differing  = ",
Count[MapThread[Unequal, {Simplify[2 cP],
{{0,0,0},{1,0,1},{1,1,0},{0,1,1}}}], True], "   (expect 1)"];

(* --- B.5  the half-turn g, and which corners are forced ------------- *)
g[p_] := {p[[2]], p[[1]], -p[[3]]} + {0,0,1};
pv = Table[First@FirstPosition[TT, g[TT[[i]]]], {i,4}];
Print["g on vertices = ", pv, "   (expect {3,4,1,2}, i.e., (AC)(BD))"];
Print["g isometry?  ", Union@Flatten@Table[
Simplify[Norm[g[TT[[i]]]-g[TT[[j]]]] - Norm[TT[[i]]-TT[[j]]]],
{i,4},{j,4}], "   (expect {0})"];
Print["g fixes the P cycle (up to rotation/reversal)? ",
MemberQ[Join[NestList[RotateLeft, #, 3] &@ {1,3,2,4},
NestList[RotateLeft, #, 3] &@ Reverse[{1,3,2,4}]],
pv[[{1,3,2,4}]]]];
Print["forced corners = ", {(TT[[1]]+TT[[3]])/2, (TT[[2]]+TT[[4]])/2},
"   on edges AC (len 1) and BD (len Sqrt[3])"];
Print["free corners lie on edges AD, BC, swapped by g."];

(* --- B.6  the invariant J ------------------------------------------- *)
KK = EllipticK[1/4];  KP = EllipticK[3/4];
JP = -(8 Pi/27) (KP/KK)^3;      (* sigma_P = 3K/K',  chi_P = -4  *)
JD = -(32 Pi/27) (KK/KP)^3;     (* sigma_D = 3K'/K,  chi_D = -16 *)
Print["J_P = ", N[JP,12], "     J_D = ", N[JD,12]];
Jof[intK_, area_] := intK volTT^2/area^3;
intK  = {-Pi/6, -Pi/2, -2 Pi/3};                 (* from B.3, Gauss-Bonnet *)
areas = {KK/(4 KP), 0.288674897531445`15, 0.292539844325784`15};
JJ = MapThread[Jof, {intK, areas}];
Print["J(disk 1,2,3) = ", N[JJ,12]];
Print["disk 1 == P, symbolically ? ", Simplify[JJ[[1]] - JP] === 0];
Print["relative gap to J_D: ", N[(JJ - JD)/JD, 4],
"   (disk 1 irrelevant; disks 2,3 exclude D)"];

(* --- B.7  the L = 2 branch, and the 2 Sqrt[3] near-miss ------------- *)
Print["chi on branch L=2  = ", Simplify[24 intK/Pi], "   (expect {-4,-12,-16})"];
Print["sigma = 12*area    = ", N[12 areas, 12]];
Print["check sigma_1 == 3K/K' ? ", Simplify[12 areas[[1]] - 3 KK/KP] === 0];
Print["sigma_2            = ", N[12 areas[[2]], 15]];
Print["2 Sqrt[3]          = ", N[2 Sqrt[3], 15]];
Print["area_2             = ", areas[[2]]];
Print["1/(2 Sqrt[3])      = ", N[1/(2 Sqrt[3]), 16]];
Print["difference         = ", N[areas[[2]] - 1/(2 Sqrt[3]), 3],
"   => sigma_2 =/= 2 Sqrt[3]"];

(* --- B.8  cross-check: J reproduces the main theorem --------------- *)
volT = 2/3;
Print["J(Sigma*_A) == J_P ? ",
Simplify[(-2 Pi/3) volT^2/(KK/KP)^3 - JP] === 0];
Print["J(Sigma*_B) == J_D ? ",
Simplify[(-Pi/3) volT^2/(KP/(2 KK))^3 - JD] === 0];

(* --- B.9  which sides of the tetrahedral four edges survive *)
tet  = {{0,0,0},{1,0,1},{1,1,0},{0,1,1}};
newc = {{0,0,0},{1,0,1},{1,1,0},{0,1,0}};
sd[p_] := Simplify[Norm[#2-#1] & @@@ Partition[p,2,1,1]];
Print["tet sides = ", InputForm[sd[tet]]];
Print["new sides = ", InputForm[sd[newc]]];
Print["sides surviving = ", Count[MapThread[Equal, {sd[tet], sd[newc]}], True],
"   (expect 2)"];
Print["vertices differing = ", Count[MapThread[Unequal, {tet, newc}], True],
"   (expect 1)"];
ang[p_] := Simplify[Table[ArcCos[Normalize[p[[Mod[k-1,4,1]]]-p[[k]]] .
Normalize[p[[Mod[k+1,4,1]]]-p[[k]]]]/Degree, {k,4}]];
Print["tet angles = ", ang[tet], "   new angles = ", ang[newc]];
Print["Count[60] = ", Count[ang[newc], 60], " <-> surviving pairs"];

(* --- B.10  classical identification: area of the Plateau disk *)
KK = EllipticK[1/4]; KP = EllipticK[3/4];
P0 = Sqrt[2] + 1;  rT = (Sqrt[2]-1)/2;        (* Smyth contour, cell inradius *)
s  = Simplify[Solve[s^2 (KK/(4 KP)) == s P0 rT/2, s][[2,1,2]]];
Print["homothety s = ", Simplify[s], "  == K'/K ? ", Simplify[s - KP/KK] === 0];
A0 = Simplify[s P0 rT/2];
Print["area in (D.6) = ", Simplify[A0], " = ", N[A0, 16]];
Print["== K'/(4K) ? ", Simplify[A0 - KP/(4 KK)] === 0];
Print["doubled: ", N[4 A0, 16], "   K'/K = ", N[KP/KK, 16]];
Print["Int K: one patch ", -Pi/6, "  vs tetrahedral ", -2 Pi/3, "   ratio ", 4];
\end{verbatim}

\begin{figure}[H]
	\vspace*{3cm}   
	\centering
	\includegraphics[width=\textwidth,height=\textheight,keepaspectratio]{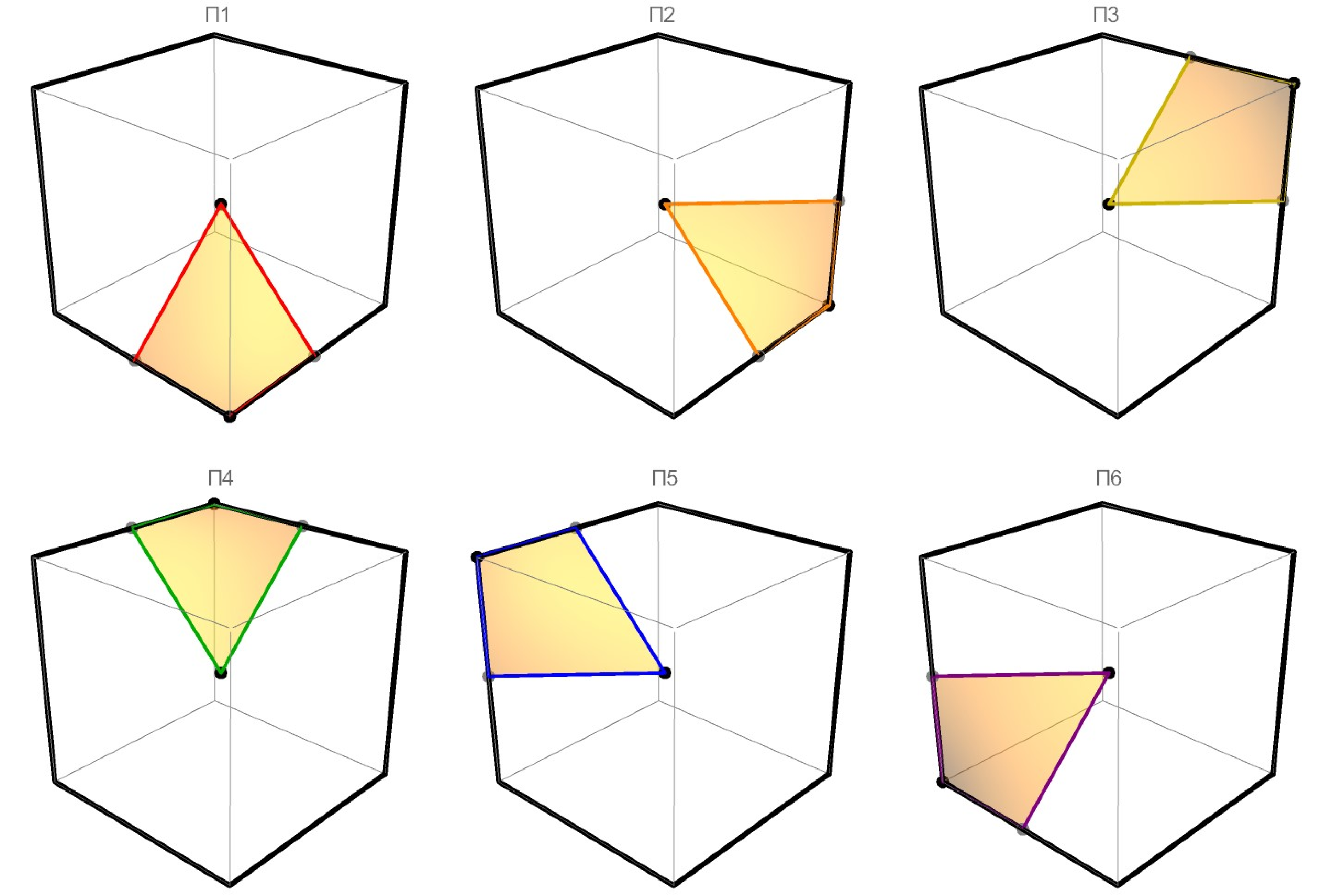}
	\caption{The six skew quadrilaterals of the cubic six edges.	
		The heavy black closed curve is Schwarz's hexagon --- the Petrie polygon
		of the cube, formed by six of its twelve edges and passing through six of
		its eight vertices; the two omitted vertices, $(0,1,0)$ and $(1,0,1)$,
		are an antipodal pair on the body diagonal that carries a $\bar3$ axis,
		and they are flat points of $\DD$. The three two-fold axes of the surface
		cut the $\DD$ disk spanning this hexagon into the six congruent patches
		$\Pi_{1},\dots,\Pi_{6}$ shown here. Each $\Pi_{k}$ has vertices $O$ (the
		cube center, black dot), the two midpoints $m_{k},m_{k+1}$ of consecutive
		hexagon edges (grey dots), and the hexagon vertex $h_{k+1}$ between them
		(black dot); its sides are therefore two interior radii of length
		$1/\sqrt2$ and two half-edges of length $1/2$, its diagonals are
		$\sqrt3/2$ and $1/\sqrt2$, and its interior angles are
		$60^{\circ},90^{\circ},90^{\circ},90^{\circ}$, the $60^{\circ}$ angle
		being at $O$ in every panel. Each $\Pi_{k}$ is thus congruent to the
		contour \eqref{eq:Qcontour}: explicitly, $\Pi_{1}$ is the image of
		\eqref{eq:Qcontour} under the half-turn
		$(x,y,z)\mapsto(-x,-y,z)+(1,\tfrac12,0)$ about the line
		$x=\tfrac12$, $y=\tfrac14$. The quadrilaterals are indexed by the
		\emph{faces} of the cube rather than its vertices --- the three vertices
		of $\Pi_{k}$ other than $O$ are coplanar in a single face, namely
		$z=0,x=1,y=1,z=1,x=0,y=0$ for $k=1,\dots,6$ --- which is why there are
		six and not eight. The shaded films are bilinear interpolants drawn to
		display the combinatorics; they are \emph{not} the minimal surfaces,
		whose exact parametrization is \cite[Eqs.~(5a)--(5c)]{GKMK}.}
	\label{fig:sixquadsd}
\end{figure}

\begin{figure}[H]
	\vspace*{3cm}   
	\centering
	\includegraphics[width=0.65\textwidth]{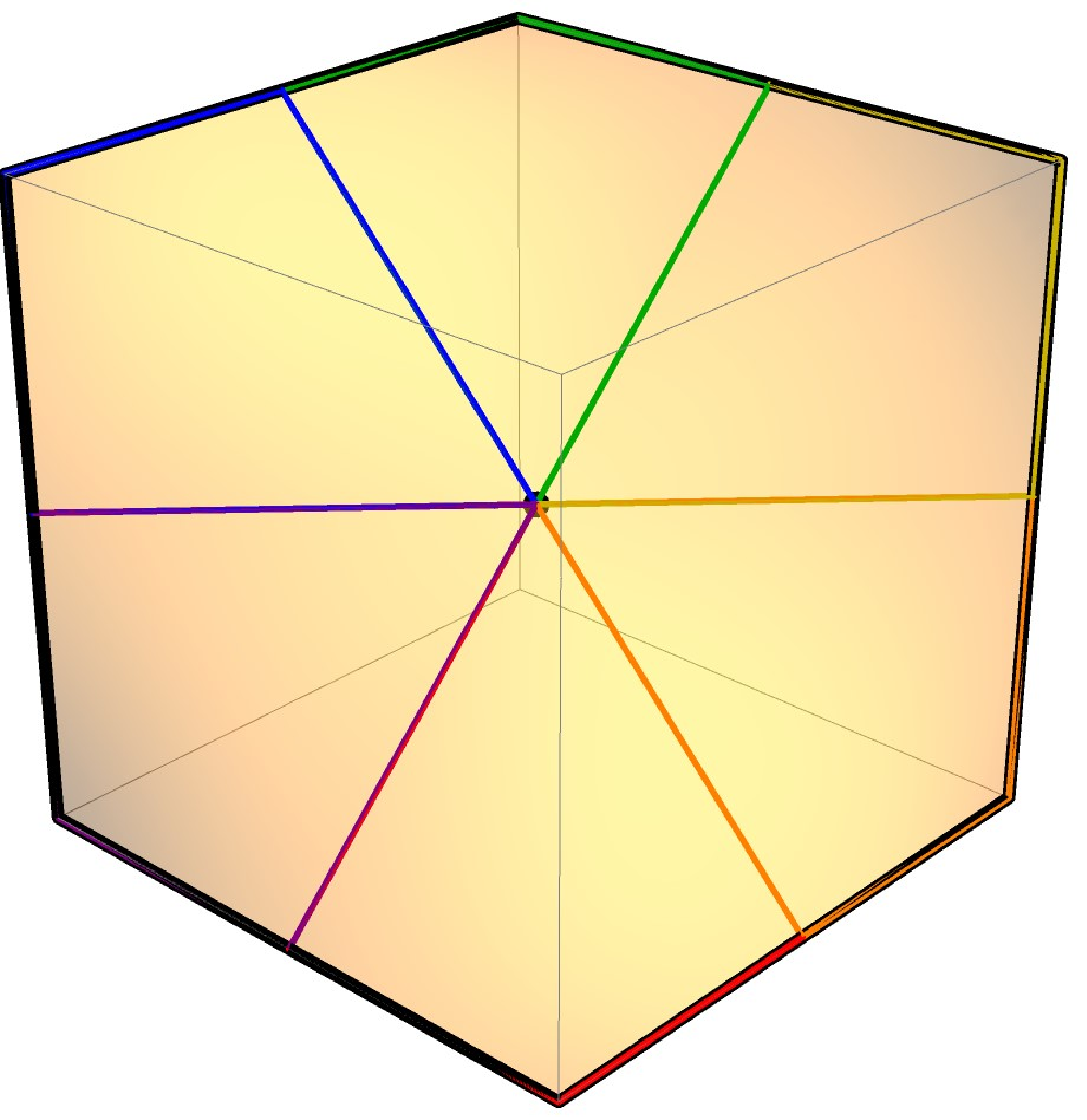}
	\caption{The same six quadrilaterals superimposed, in the colors of
		Figure~\ref{fig:sixquadsd}. All six meet at the cube center, each
		contributing its $60^{\circ}$ angle, for a total of $360^{\circ}$; each
		of the six hexagon edges is halved and shared between two neighbors.
		Since every boundary arc is straight, hence geodesic, Gauss--Bonnet gives
		$\int_{\Pi_{k}}\!K\,dA=-2\pi+\tfrac{11\pi}{6}=-\tfrac{\pi}{6}$ for each
		patch, and $-\pi$ for the six-patch piece spanning the hexagon. Again the
		films are bilinear stand-ins, not minimal surfaces.}
	\label{fig:sixquadsd1}
\end{figure}

\begin{figure}[H]
	\vspace*{1cm}   
	\centering
	\includegraphics[width=\textwidth,height=\textheight,keepaspectratio]{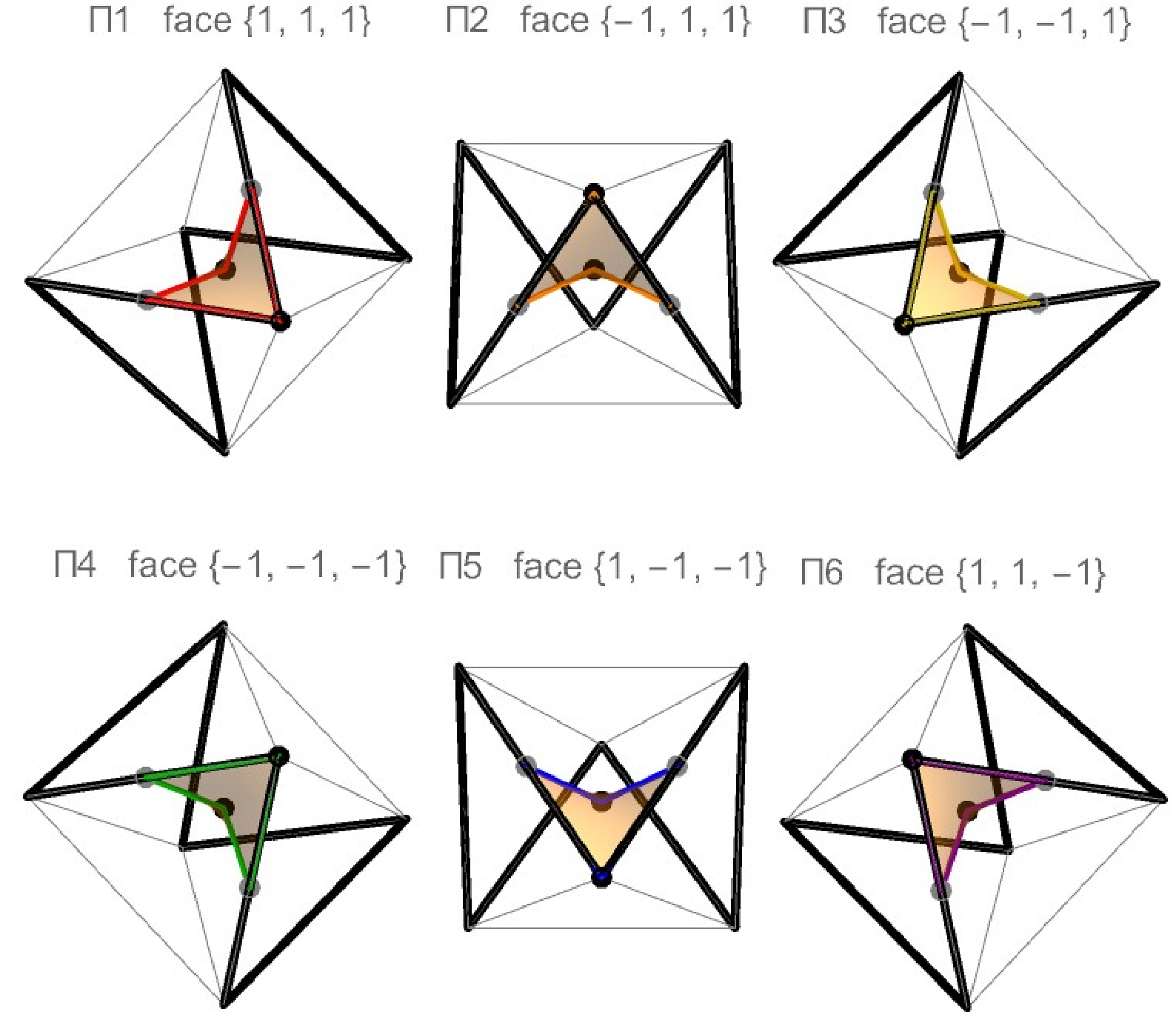}
	\caption{The six skew quadrilaterals of the octahedral six edges.
		The heavy black closed curve is Schwarz's second hexagon --- the Petrie
		polygon of the regular octahedron, formed by six of its twelve edges and
		passing through \emph{all} six of its vertices. The three two-fold axes
		of the surface cut the $\PP$ disk spanning this hexagon into the six
		congruent patches $\Pi_{1},\dots,\Pi_{6}$ shown here. Each $\Pi_{k}$ has
		vertices $O$ (the octahedron center, black dot), the two midpoints
		$m_{k},m_{k+1}$ of consecutive hexagon edges (grey dots), and the hexagon
		vertex $h_{k+1}$ between them (black dot). Here the two interior radii
		and the two half-edges have the \emph{same} length $1/\sqrt2$, so all
		four sides are equal; the diagonals are $1$ and $1/\sqrt2$, and the
		interior angles alternate $60^{\circ},90^{\circ},60^{\circ},90^{\circ}$,
		with a $60^{\circ}$ angle at $O$ and at $h_{k+1}$. Each $\Pi_{k}$ is
		therefore congruent to the octahedral four edges $\Oct$ of
		\S\ref{subsec:terminology} --- a coincidence of shape, not of structure,
		since $\Oct$ is built from four \emph{whole} octahedron edges whereas
		$\Pi_{k}$ is built from two half-edges and two interior radii. The
		quadrilaterals are indexed by the \emph{faces} of the octahedron: the
		three vertices of $\Pi_{k}$ other than $O$ are coplanar in a single face
		$\pm x\pm y\pm z=1$, and the six faces so used omit the antipodal pair
		$-x+y-z=1$, $x-y+z=1$, whose normals lie along the $\langle111\rangle$
		body diagonal carrying a $\bar3$ inversion axis --- the direction of the
		flat points of $\PP$. This is why there are six and not eight. The shaded
		films are bilinear interpolants drawn to display the combinatorics; they
		are \emph{not} the minimal surfaces, whose exact parametrization is
		\cite[Eqs.~(4a)--(4c)]{GK}. Note the scale: $\Pi_{k}\cong\Oct$ exactly, 
	    whereas the corresponding Smyth contour of the disphenoid $T$ is $2\cdot\Oct$
        (Lemma~\ref{lem:contour}(b)).}
	\label{fig:sixquadsp}
\end{figure}

\begin{figure}[H]
	\vspace*{2cm}   
	\centering
	\includegraphics[width=0.65\textwidth]{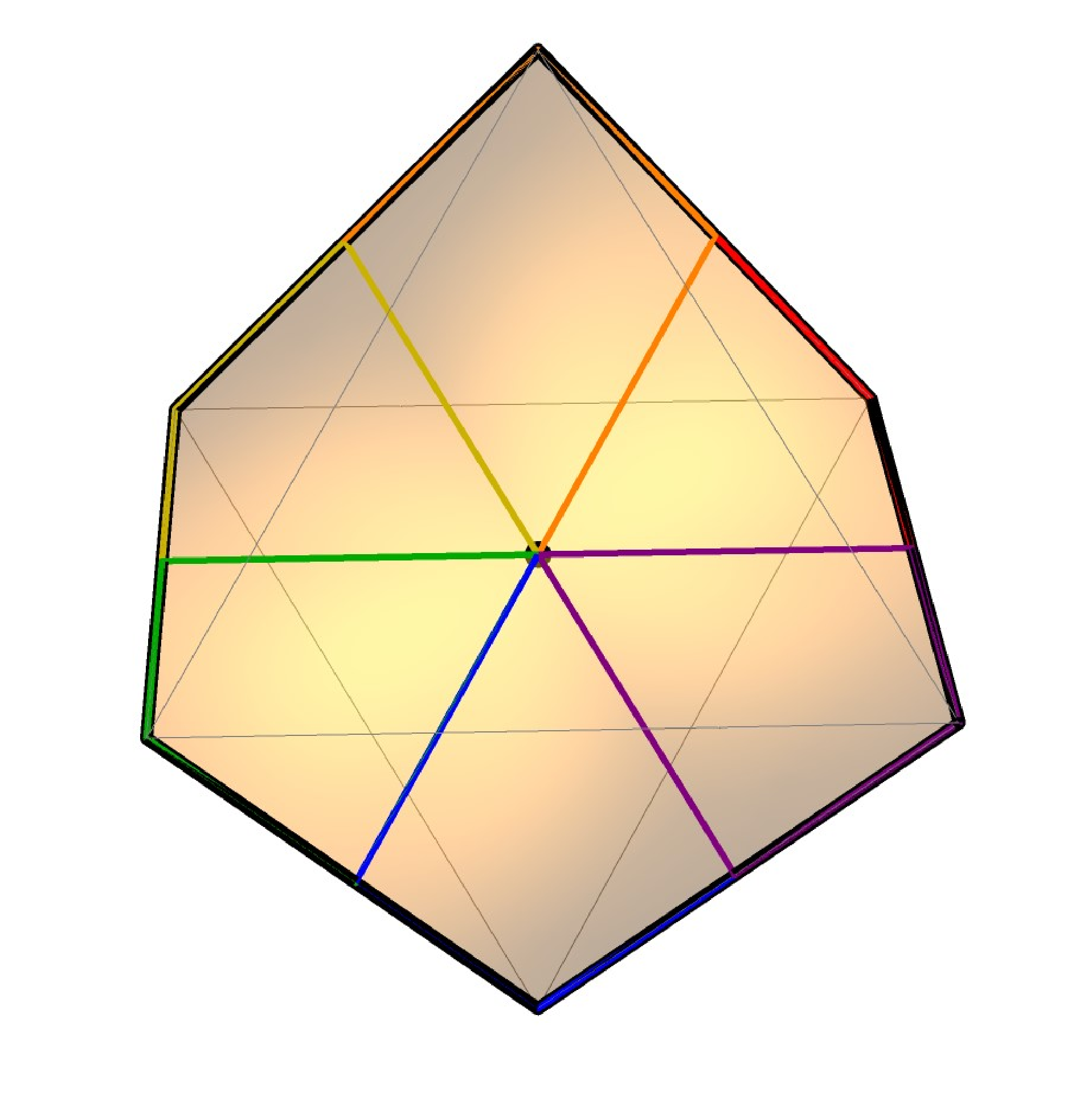}
	\caption{The same six quadrilaterals superimposed, in the colors of
		Figure~\ref{fig:sixquadsp}. All six meet at the octahedron center, each
		contributing its $60^{\circ}$ angle, for a total of $360^{\circ}$; each
		of the six hexagon edges is halved and shared between two neighbors.
		Every boundary arc is straight, hence geodesic, so Gauss--Bonnet gives
		$\int_{\Pi_{k}}\!K\,dA=-2\pi+\tfrac{5\pi}{3}=-\tfrac{\pi}{3}$ for each
		patch, and $-2\pi$ for the six-patch piece spanning the hexagon. Since a
		fundamental patch of $\PP$ or $\DD$ carries $-\pi/6$, each $\Pi_{k}$ here
		comprises \emph{two} fundamental patches, whereas each quadrilateral of
		Figure~\ref{fig:sixquadsd} comprises exactly one; the octahedral hexagon
		piece is thus twice as large, in patch count, as the cubic one. Again the
		films are bilinear stand-ins, not minimal surfaces.}
	\label{fig:sixquadsp1}
\end{figure}
\ \

\section*{Errata and addenda}
Brakke's numerical estimate $0.288674897531445$ of the Schoen I-WP area was mistaken (and easily corrected); in fact, an exact expression is known \cite{CK}:
\[               
\tfrac16 \sqrt3 \approx 0.288675134594813.
\]
The Neovius area likewise possesses an exact expression \cite{ADSN}:
\[
\tfrac14 K\!\left(\tfrac23\right) / K\!\left(\tfrac13\right) \approx 0.292539844325784
\]
and thus the two ``other'' Smyth disks in a quadrirectangular tetrahedron are well understood.  But a proof of the latter result does not appear in \cite{ADSN} nor seemingly elsewhere in the literature, as far as is known.

\clearpage

\subsection*{Acknowledgements}
I thank Ken Brakke for Surface Evolver computations and for his
disphenoid pages \cite{BrakkeDisph}, which supplied the decisive hint.
Anthropic Claude Opus 5.0 played a crucial role in turning 
my conjectures into theorems: to say that the AI model merely ``assisted'' 
would be an understatement.  It was used in an absolutely essential research 
capacity throughout every aspect of the paper.  No AI funding whatsoever for 
this study was received.  My interest in minimal surfaces began years ago 
(see \cite{FinchDP}).

\end{document}